\documentclass[10pt,reqno]{amsart}
\usepackage{graphicx}  
\usepackage{enumerate}
\usepackage[colorlinks=true]{hyperref}
\hypersetup{urlcolor=blue, citecolor=red}

\usepackage{mathabx} 
\usepackage[shortlabels]{enumitem} 

\usepackage{bm} 

\theoremstyle{plain}
\newtheorem{theorem}{Theorem}[section]
\newtheorem{proposition}[theorem]{Proposition}
\newtheorem{lemma}[theorem]{Lemma}
\theoremstyle{definition}
\newtheorem{remark}{Remark}
\newtheorem{example}{Example}[section]

\title[A Criterion for the Existence of Relaxation Oscillations]%
{A Criterion for the Existence of \\
Relaxation Oscillations
with Applications to 
\\
Predator-Prey Systems
and an Epidemic Model}

\author[T.-H.~Hsu and G.S.K.~Wolkowicz]{}

\email{hsut1@math.miami.edu}
\email{wolkowic@mcmaster.ca}

\date{\today}



\begin{document}

\maketitle

\centerline{\scshape Ting-Hao~Hsu$^{\dag}$ and Gail~S.~K.~Wolkowicz$^{\ddag}$}
\medskip
{\footnotesize
\centerline{$^\dag$Department of Mathematics}
\centerline{University of Miami}
\centerline{1365 Memorial Drive}
\centerline{Coral Gables, FL 33146, USA}
\medskip
\centerline{$^\ddag$Department of Mathematics and Statistics}
\centerline{McMaster University}
\centerline{1280 Main Street West}
\centerline{Hamilton, Ontario L8S 4K1, Canada}
}
\bigskip

\begin{abstract}
We derive characteristic functions
to determine the number and stability of relaxation oscillations
for a class of planar systems.
Applying our criterion,
we give conditions under which
the chemostat predator-prey system
has a globally orbitally asymptotically stable limit cycle.
Also we demonstrate that
a prescribed number of relaxation oscillations
can be constructed
by varying the perturbation
for an epidemic model
studied by Li et al.\ [SIAM J. Appl. Math, 2016].
\end{abstract}

\newcommand{\IN}{\mathrm{in}}
\newcommand{\OUT}{\mathrm{out}}

\section{Introduction}
Periodic orbits in ecological models are important because
they can be used to explain oscillatory phenomena observed in real-world data.
Relaxation oscillations
are periodic orbits
formed from slow and fast sections.
In this paper,
we extend the criterion for the existence
of relaxation oscillations given by Hsu \cite{Hsu:2019}
to a class of planar systems.
We apply our criterion to two ecological models:
a predator-prey system in a chemostat,
and an epidemic model.
Both systems are three-dimensional,
and have two-dimensional invariant manifolds.

Predator-prey interaction in a well-stirred chemostat
(see e.g.\ \cite{Smith:1995})
can be modeled
by the following ordinary differential equations:
\begin{equation}\label{deq_predprey}\begin{aligned}
  &\dot{S}= (S^0-S)\epsilon- \rho mSx,
  \\
  &\dot{x}= x\big(-\epsilon+mS\big)-cyp(x),
  \\
  &\dot{y}= y\big(-\epsilon+p(x)\big),
\end{aligned}\end{equation} where $\cdot$ denotes $\frac{d}{dt}$,
$S(t)$ is the concentration of the nutrient in the growth chamber at time $t$,
$x(t)$ and $y(t)$ are the density of prey (which feeds off this nutrient)
and predator populations, respectively;
$S^0$ denotes the concentration of the input nutrient,
and $\rho$ and $c$ are
constants related to the consumption of the nutrient by the prey
population
and the consumption of the prey by the predators, respectively.
To ensure that the volume of this vessel remains constant,
$\epsilon$ denotes both the rate of inflow from the nutrient reservoir to the growth chamber,
as well as the rate of outflow from the growth chamber.
The functional response $p(x)$,
which describes the change in the density of the prey attacked per unit time per predator,
is continuously differentiable and satisfies
\begin{equation}\label{cond_p}
  p(0)=0,\;\;
  p'(0)>0,\;\;
  \text{and}\;\;
  p(x)>0\;\;\forall\;x>0.
\end{equation}
It can be verified that
system \eqref{deq_predprey}
under assumption \eqref{cond_p}
has a unique positive equilibrium for all small $\epsilon>0$.
From the equations in \eqref{deq_predprey},
\begin{equation}\label{decay_predprey}
  \frac{d}{dt}
  \left(S+\rho x+ c\rho y\right)
  = -\epsilon\left(S+\rho x+ c\rho y-S^0\right),
\end{equation} so system \eqref{deq_predprey} has an invariant simplex
\begin{equation}\label{def_Lambda_predprey}
  \Lambda= \{(S,x,y)\in \mathbb R^3_+: S+\rho x+c\rho y=S^0\}
\end{equation}
that attracts all points in $\mathbb R^3_+$.
For system \eqref{deq_predprey} with $\epsilon=0$,
there is a continuous family of heteroclinic orbits on $\Lambda$
as illustrated in Figure~\ref{fig_predprey_H2gamma}
(see equation \eqref{fast_predprey} in Section \ref{sec_predprey}
and its succeeding paragraph for the limiting system on $\Lambda$).
Each heteroclinic orbit 
connects two points
on the boundary of $\Lambda$,
where $x=0$.
On the other hand,
the restriction of system \eqref{deq_predprey} on the plane $\{x=0\}$ is \begin{equation}\label{slow_predprey}
  \dot{S}=(S^0-S)\epsilon,
  \quad
  \dot{x}=0,
  \quad
  \dot{y}=-\epsilon y.
\end{equation}
Hence the  segment $\overline{\Lambda}\cap \{x=0\}$
is a trajectory of system \eqref{deq_predprey}
approaching the positive $S$-axis.
The heteroclinic orbits for the limiting system
and the trajectory for \eqref{slow_predprey},
forms a continuous family of closed loops.
In Section~\ref{sec_predprey},
we use these loops to construct periodic orbits for the full system
(see Theorem \ref{thm_predprey}).
Under certain conditions,
we show that \eqref{deq_predprey}
has a globally asymptotically periodic orbit in $\mathbb R^3_+$
for all sufficiently small $\epsilon>0$.
Moreover, the minimal period of the periodic orbit is of order $1/\epsilon$
as $\epsilon\to 0$,
and the trajectory converges to one of the closed loops described above.
That is, this family of periodic orbits forms a relaxation oscillation.

\begin{figure}[htb]
{\centering
{\includegraphics[trim= 0cm 1cm 0cm 0cm, clip, width=.7\textwidth]{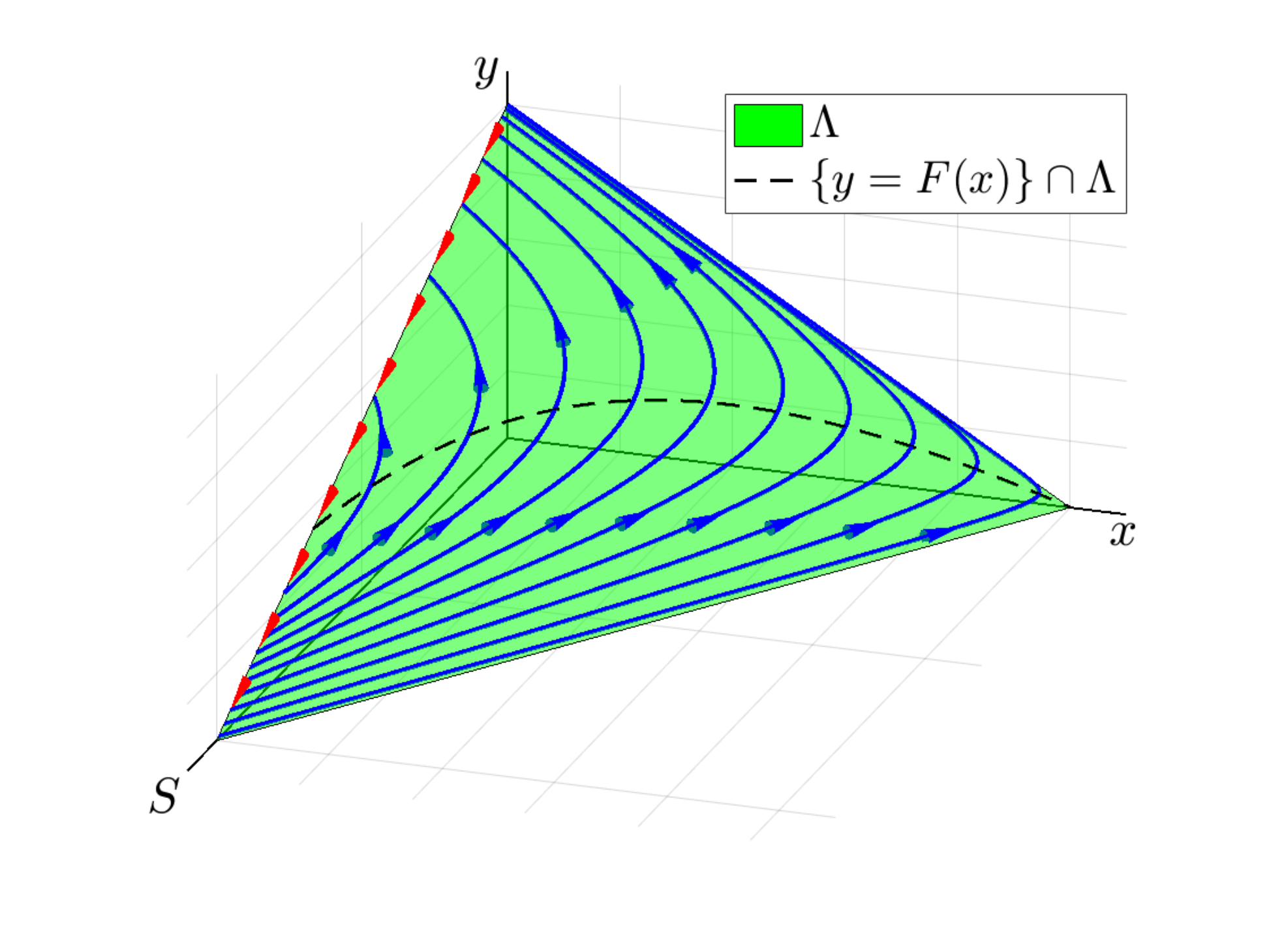}}
\par}
\caption{
System \eqref{deq_predprey} with $\epsilon=0$
exhibits a family of heteroclinic orbits on $\Lambda$.
The dynamics on the segment $\overline{\Lambda}\cap \{x=0\}$
is governed by system \eqref{slow_predprey}.
}
\label{fig_predprey_H2gamma}
\end{figure}

We also study the epidemic model
\begin{equation}\label{deq_SIR}\begin{aligned}
  &\dot{S}= b(N)- g(S,N)I- (D+p)S,
  \\
  &\dot{I}=g(S,N)I- (d+\gamma+\alpha)I,
  \\
  &\dot{R}= pS+\gamma I-DR,
\end{aligned}\end{equation}
which was proposed and investigated by Graef et al.\ \cite{Graef:1996}
and Li et al.\ \cite{Li:2016}.
Here $N=S+I+R$ and $b(N)=DN+\epsilon f(N)$,
with \begin{equation}\label{SIN_f}
  f(N)= rN\left(1-\frac{N}{N_{\max}}\right),
\end{equation}
and \begin{equation}\notag
  \partial_Ng(S,N)>0
  \quad\text{and}\quad
  \partial_Ng(S,N)>0
  \quad\forall\; S\ge 0,\; N\ge 0.
\end{equation}
Let $a=d+\gamma+\alpha$.
System \eqref{deq_SIR} is equivalent to
\begin{equation}\label{deq_SIN}\begin{aligned}
  &\dot{S}= DN+ \epsilon f(N)- g(S,N)I- (D+p)S,
  \\
  &\dot{I}=g(S,N)I- aI,
  \\
  &\dot{N}= \epsilon f(N)-\alpha I.
\end{aligned}\end{equation}
Setting $\epsilon=0$ in \eqref{deq_SIN}, we obtain the limiting system \begin{equation}\label{deq_center_SIN}\begin{aligned}
  &\dot{S}= DN+ g(S,N)I- (D+p)S,
  \\
  &\dot{I}=g(S,N)I- aI,
  \\
  &\dot{N}= -\alpha I.
\end{aligned}\end{equation}
Note that the line $\mathcal{Z}_0\equiv\{(S,I,N): I=0, S=\frac{D}{D+p}N\}$
is a set of equilibria of \eqref{deq_SIN}
in the invariant plane $\{I=0\}$.
It is known \cite{Graef:1996,Li:2016} that
$\mathcal{Z}_0$ consists of the endpoints of a family of heteroclinic orbits
(see Figure~\ref{fig_SIN_gamma}).
The existence of periodic orbits of \eqref{deq_SIN}
was proved by Li et.\ al \cite{Li:2016}.
In Section~\ref{sec_SIN},
we demonstrate that
a prescribed number of relaxation oscillations for system \eqref{deq_SIN} 
can be obtained by varying the perturbation term $\epsilon f(N)$.

\begin{figure}[htb]
{\centering
{\includegraphics[trim = 0cm 1cm 0cm .5cm, clip, width=.7\textwidth]{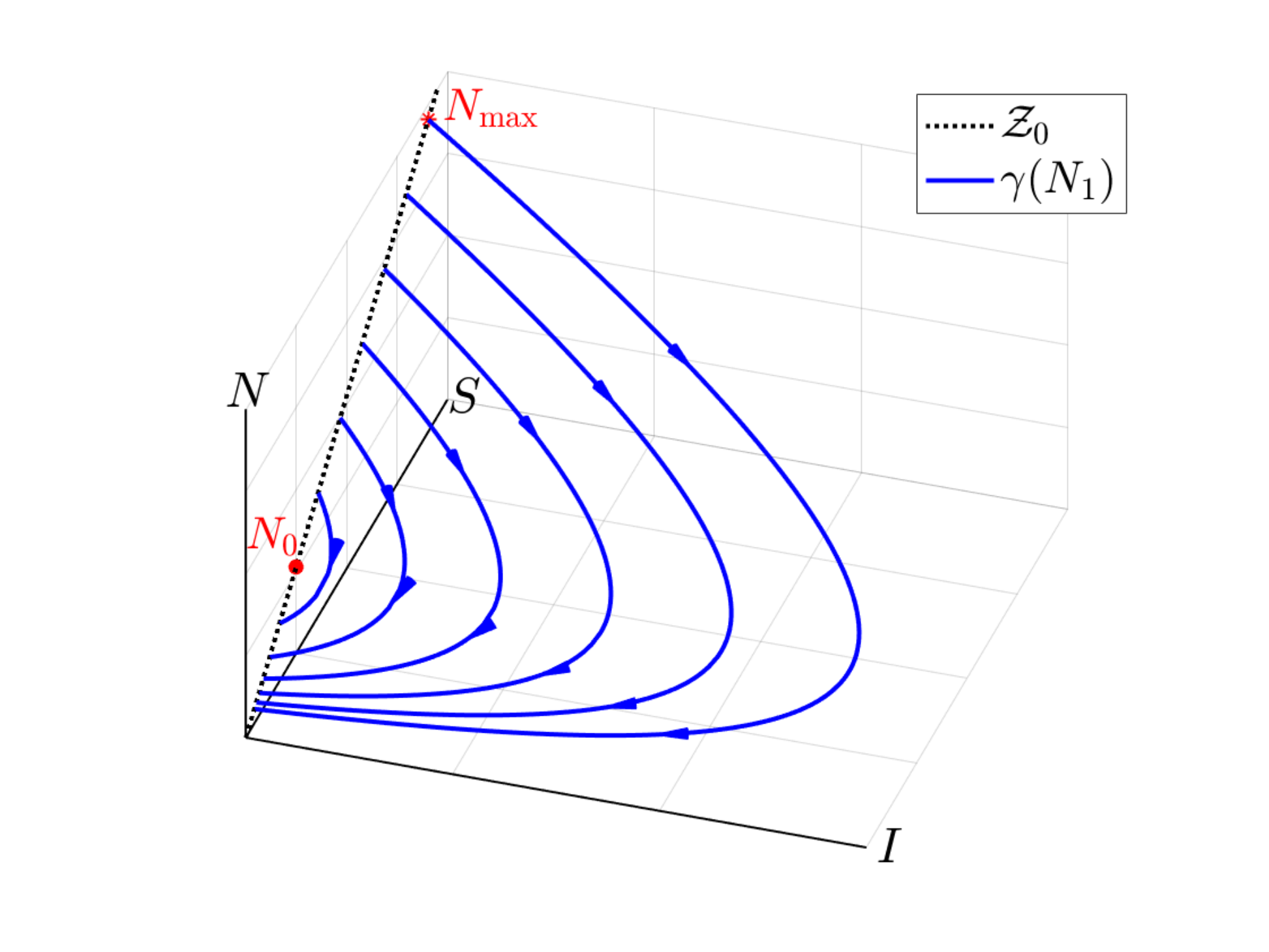}}
\par}
\caption{
Trajectories for the limiting system \eqref{deq_center_SIN}
of \eqref{deq_SIN} with $\epsilon=0$.
}
\label{fig_SIN_gamma}
\end{figure}

Singular perturbations in predator-prey systems were studied in various contexts.
For a model of two predators competing for the same prey,
when the prey population grows much faster than the predator populations,
the existence of a relaxation oscillation
was proved by Liu, Xiao and Yi
\cite{Liu:2003}.
For predator-prey systems
with Holling type III or IV,
when the death and the yield rates of the predator
are small and proportional to each other,
the canard phenomenon
and the cyclicity of limit cycles
were investigated by 
Li and Zhu \cite{Li:2013}.
For a class of
the Holling-Tanner model,
when the intrinsic growth rate of the predator is sufficiently small,
the existence of a relaxation oscillation
was proved by
Ghazaryan, Manukian and Schecter
\cite{Ghazaryan:2015}.
For predator-prey models
with eco-evolutionary dynamics
in which ecological and evolutionary interactions
occur on different time scales,
relaxation oscillations
were investigated by
Piltz et~al.\ \cite{Piltz:2017}
and Shen~et~al.~\cite{Shen:2019}.
For the classical predator-prey model
with Monod functional response
and small predator death rates,
the  unique periodic orbit,
which was proved to exist by 
Liou and Cheng \cite{Liou:1988}
(who corrected a flaw in the original proof of Cheng \cite{Cheng:1981})
and Kuang and Freedman \cite{Kuang:1988},
was proved to form a relaxation oscillation
by Hsu and Shi \cite{Hsu:2009},
Wang et al.\ \cite{Wang:2014},
and Lundstr\"{o}m and S\"{o}derbacka \cite{Lundstrom:2018},
and the cyclicity of the limit cycle
was investigated by Huzak \cite{Huzak:2018}.
For general functional responses,
the number 
of relaxation oscillations
depending on
the number of local extrema of the prey-isocline
was studied by Hsu \cite{Hsu:2019}.

This paper is organized as follows.
In Section~\ref{sec_criteria},
we state and prove criteria for
the location and stability of relaxation oscillations
in a class of planar systems.
In Section~\ref{sec_predprey}
we investigate our criterion
for the chemostat predator-prey system \eqref{deq_predprey},
and give conditions under which
the system has exactly one or two periodic orbits.
The epidemic model \eqref{deq_SIN} is studied in Section~\ref{sec_SIN},
in which we compute the characteristic functions
in terms of a parametrization of the center manifold,
and we use numerical simulation
to find the number of relaxation oscillations.

\section{A Criterion for Relaxation Oscillations in Planar Systems}
\label{sec_criteria}

To determine the location and stability
of relaxation oscillations
for predator-prey systems
with small predator death,
a criterion was given in \cite{Hsu:2019}.
In this section
we extend that criteria
by considering planar systems of the form
\begin{equation}\label{deq_ab}\begin{aligned}
  &\dot{a}= \epsilon f(a,b,\epsilon)+ b\, h(a,b,\epsilon),
  \\
  &\dot{b}= b\,g(a,b,\epsilon),
\end{aligned}\end{equation}
where $f$, $g$ and $h$ are smooth functions
satisfying,
for some numbers $-\infty\le a_{\min}< \bar{a}< a_{\max}\le \infty$,
\begin{align}
  \label{turning_f}
  &f(a,0,0)> 0
  \quad
  \quad\forall\; a\in (a_{\min},a_{\max}),
  \\[.5em]
  \label{turning_h}
  &
  h(a,b,0)<0
  \quad\forall\; a\in (a_{\min},a_{\max}),\; b\ge 0,
  \\
\intertext{and}
  \label{turning_g}
  &g(a,0,0)\begin{cases}
    <0,&\text{if }a\in (a_{\min},\bar{a}),
    \\
    >0,&\text{if }a\in (\bar{a},a_{\max}).
  \end{cases}
\end{align}

Setting $\epsilon=0$ in \eqref{deq_ab},
we obtain the limiting system
\begin{equation}\label{fast_ab}\begin{aligned}
  &\dot{a}= b\, h(a,b,0),
  \\
  &\dot{b}= b\, g(a,b,0).
\end{aligned}\end{equation}
We assume the following condition
(see Figure~\ref{fig_gamma}):

\begin{enumerate}[(H)]
\item
\label{cond_H}
There exists a nonempty open interval $I$
and smooth functions $a_\alpha:I\to (\bar{a},a_{\max})$
and $a_\omega:I\to (a_{\min},\bar{a})$,
such that, for each $s\in I$,
the points $(a_\alpha(s),0)$ and $(a_\omega(s),0)$
are the alpha- and omega-limit points, respectively,
of a trajectory of \eqref{fast_ab}.
\end{enumerate}

\begin{figure}[htb]
{\centering
{\includegraphics[trim = 2.7cm .9cm 1cm .7cm, clip, width=.55\textwidth]{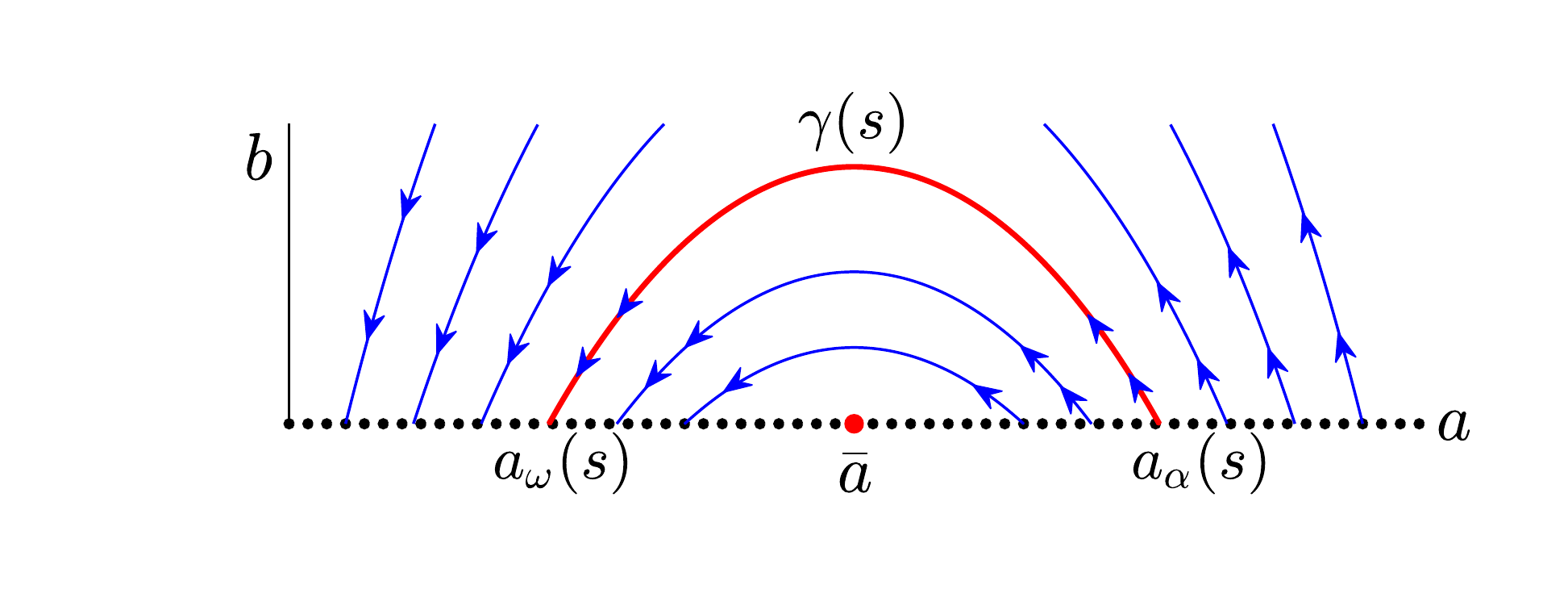}}
\par}
\caption{
A family of heteroclinic orbits,
parameterized by $\gamma(s)$,
of limiting system \eqref{fast_ab}.
}
\label{fig_gamma}
\end{figure}

Our approach is to use
{\em geometric singular perturbation theory}
\cite{Fenichel:1979,Jones:1995,Kuehn:2016}
to construct periodic orbits.
The idea is that solutions of the full system
can potentially be obtained
by joining trajectories of the limiting systems.
Under assumption $\mathrm{(\hyperref[cond_H]{H})}$
for each $s\in I$,
we denote $\gamma(s)$ the trajectory of \eqref{deq_ab}
connecting $(a_\alpha(s))$ and $(a_\omega(s))$,
and denote $\sigma(s)$ the segment in the $a$-axis
that starts at $(a_\omega(s),0)$ and ends $(a_\alpha(s),0)$.
Then $\Gamma(s)=\gamma(s)\cup \sigma(s)$
forms a closed loop.
Hence each $\Gamma(s)$
is potentially the limiting configuration of a relaxation oscillation.
Using a variation of
{\em bifurcation delay}
(see \cite{DeMaesschalck:2016,Hsu:2017}
and the references therein),
in the following theorem we are able to show that
only certain $\Gamma(s)$
correspond to relaxation oscillations.

\begin{theorem}\label{thm_general}
Assume \eqref{turning_f}--\eqref{turning_g}
and $\mathrm{(\hyperref[cond_H]{H})}$.
Set \begin{equation}\label{def_chi}
  \chi(s)
  = \int_{a_\omega(s)}^{a_\alpha(s)} \frac{g(a,0,0)}{f(a,0,0)}\;da
\end{equation}
and \begin{equation}\label{def_lambda}
  \lambda(s)
  = \ln\frac{f(a_\alpha(s),0,0)}{f(a_\omega(s),0,0)}
  + \int_{\gamma(s)}\frac{\partial_ah(a,b,0)}{h(a,b,0)}\;da
  + \int_{\gamma(s)}\frac{\partial_bg(a,b,0)}{g(a,b,0)}\;db.
\end{equation}
If $s_0\in I$ satisfies $\chi(s_0)=0$ and $\lambda(s_0)\ne 0$,
then for all sufficiently small $\epsilon>0$,
there is a periodic orbit ${\ell}_\epsilon$ of \eqref{deq_ab}
in a $O(\epsilon)$-neighborhood of $\Gamma(s_0)$.
The minimal period of $\ell_\epsilon$, denoted by $T_\epsilon$, satisfies \begin{equation}\label{est_Teps_ab}
  T_\epsilon= \frac{1}{\epsilon}\left(
    \int_{a_\omega(s_0)}^{a_\alpha(s_0)}\frac{1}{f(a,0,0)}\;da+o(1)
  \right)
  \quad\text{as }\epsilon\to 0.
\end{equation}
Moreover, 
${\ell}_\epsilon$ is locally orbitally asymptotically stable if $\lambda(s_0)<0$,
and is orbitally unstable if $\lambda(s_0)>0$.
Conversely, 
if $\chi(s_0)\ne 0$,
then for any point $z_1$ in the interior of the trajectory $\gamma(s_0)$,
there is a neighborhood $U$ of $z_1$
such that no periodic orbit of \eqref{deq_ab}
intersects $U$ for any sufficiently small $\epsilon>0$.
\end{theorem}

\begin{remark}
\label{rmk_gh}
Condition \eqref{turning_h} is not essential
and was only provided for convenience.
Without assuming the positivity of $h$,
the results in Theorem \ref{thm_general}
still hold with
$\lambda$ defined 
using integrals in terms of the time variable $t$.
Also note that all integrals in \eqref{def_chi} and \eqref{def_lambda}
exist and are finite
under the assumption that $f$ and $h$ remain nonzero.
To see that the last integral of \eqref{def_lambda},
one way is to use the relation
$db/da=g(a,b,0)/h(a,b,0)$ from \eqref{fast_ab}
to obtain \begin{equation}\label{int_gh}
  \int_{\gamma(s)}\frac{\partial_bg(a,b,0)}{g(a,b,0)}\;db
  =\int_{\gamma(s)}\frac{\partial_bg(a,b,0)}{h(a,b,0)}\;da.
\end{equation}
This integral is finite
since the integrand in the last expression is bounded.
\end{remark}

\begin{remark}
The function $\chi$ defined by \eqref{def_chi}
is related to the {\em slow divergence integral}
studied by De Maesschalck and Dumortier \cite{DeMaesschalck:2010}.
In their work, the cyclicity of a relaxation oscillation
near $\gamma(s_0)$
is bounded by an algebraic expression of
the the multiplicity of $s_0$
as a zero of $\chi(s)$.
The function $\lambda(s)$
in the present work
is not necessarily equivalent to $\chi'(s)$
(in the sense of multiplication by a positive function).
When $\chi(s_0)=0$,
it can be shown that $\lambda(s_0)$ and $\chi'(s_0)$
have the same sign.
Therefore, the function $\lambda$
provides an essentially different approach
than that in \cite{DeMaesschalck:2010}
for determining the sign of $\chi'(s_0)$.
\end{remark}

\begin{remark}
The function $\chi(s)$ defined by \eqref{def_chi}
can be expressed in terms of integrals on $\gamma(s)$:
For any fixed $s\in I$,
let $(a,b)=(A(t),B(t))$
be a solution of \eqref{fast_ab} with trajectory $\gamma(s)$.
Then the relations $A(\infty)=a_{\omega}(s)$
and $A(-\infty)=a_{\alpha}(s)$ yield \begin{equation}\notag
  \int_{a_\omega(s)}^{a_\alpha(s)} \frac{g(a,0,0)}{f(a,0,0)}\;da
  =\int_{-\infty}^{\infty} \frac{g(a,0,0)}{f(a,0,0)}\;A'(t)\;dt
  =\int_{-\infty}^{\infty} \frac{g(a,0,0)}{f(a,0,0)}\;\frac{A'(t)}{B'(t)}B'(t)\;dt,
\end{equation}
and therefore \begin{equation}\notag
  \chi(s)
  = \int_{\gamma(s)} \frac{g(a,0,0)}{f(a,0,0)}\;da
  = \int_{\gamma(s)} \frac{g(a,0,0)}{f(a,0,0)}\;
  \frac{g(a,b,0)}{h(a,b,0)}\;db.
\end{equation}
These last two identities sometimes can be advantageous
for determining the sign of
$\chi(s)$
(see Theorem \ref{thm_1hump}).
\end{remark}

\begin{remark}
For the perspective of modeling,
if a planar system of the form \eqref{fast_ab}
possesses a line of equilibria
and a family of heteroclinic orbits connecting points on this line,
then by a suitable choice of $g(a,b)$
to control the signs of $\chi$ and $\lambda$
in \eqref{def_chi} and \eqref{def_lambda},
the perturbed system \eqref{deq_ab}
can have a relaxation oscillation at a prescribed location
with a prescribed local stability
(see Example \ref{ex_SIN2}).
\end{remark}

\begin{remark}
In the case that $\chi(s_0)=0$ and $\lambda(s_0)=0$,
Theorem~\ref{thm_general}
does not guarantee the existence of a periodic orbit near $\gamma(s_0)$.
However,
in the region filled by the family of heteroclinic orbits
given in $\mathrm{(\hyperref[cond_H]{H})}$,
by Theorem~\ref{thm_general}
all possible periodic orbits
must lie in
an arbitrarily small neighborhood
of $\bigcup_{\{s_0: \chi(s_0)=0\}}\gamma(s_0)$
regardless of the sign of $\lambda(s_0)$
for all sufficiently small $\epsilon>0$.
\end{remark}

\begin{proposition}
\label{prop_lambdaFH}
Assume the functions $f(a,b,\epsilon)$ and $h(a,b,\epsilon)$ in \eqref{deq_ab}
are separable functions of the form
\begin{equation}\label{fh_seperable}
  f(a,b,\epsilon)= a\,\tilde{f}(b,\epsilon)
  \;\;\text{and}\;\;
  h(a,b,\epsilon)= a\,\tilde{h}(b,\epsilon).
\end{equation}
Then $\lambda(s)$ defined in \eqref{def_lambda} equals
\begin{equation}\label{lambda_simple1}
  \lambda(s)
  = \int_{\gamma(s)}\frac{\partial_bg(a,b,0)}{g(a,b,0)}\;db.
\end{equation}
\end{proposition}

\begin{proof}
Under condition \eqref{fh_seperable},
\begin{equation}\label{ln_ff}
  \ln \frac{f(a_\alpha(s),0,0)}{f(a_\omega(s),0,0)}
  =\ln \frac{a_\alpha(s)}{a_\omega(s)}
\end{equation}
and \begin{equation}\label{int_ahh}
  \int_{\gamma(s)}\frac{\partial_ah(a,b,0)}{h(a,b,0)}\;da
  =\int_{\gamma(s)}\frac{1}{a}\;da
  =\ln \frac{a_\omega(s)}{a_\alpha(s)}.
\end{equation}
Substitute \eqref{ln_ff} and \eqref{int_ahh} into \eqref{def_lambda},
we obtain \eqref{lambda_simple1}.
\end{proof}

\begin{remark}
Proposition~\ref{prop_lambdaFH}
is the case considered in \cite{Hsu:2019}.
\end{remark}

\begin{proposition}
\label{prop_lambdaH}
Assume that the function $h(a,b,\epsilon)$ in \eqref{deq_ab}
is independent of $a$, that is,
\begin{equation}\label{h_independent}
  h(a,b,\epsilon)= \tilde{h}(b,\epsilon).  
\end{equation}
Then $\lambda(s)$ defined in \eqref{def_lambda} equals
\begin{equation}\label{lambda_simple2}
  \lambda(a_0)
  = \ln\frac{f(a_\alpha(s),0,0)}{f(a_\omega(s),0,0)}
  + \int_{\gamma(a_0)}\frac{\partial_bg(a,b,0)}{g(a,b,0)}\;db.
\end{equation}
\end{proposition}

\begin{proof}
The partial derivative $\partial_ah$
is identically zero under assumption \eqref{h_independent}.
Hence the second integral in \eqref{def_lambda} is zero.
\end{proof}

\begin{proposition}
\label{prop_lambdaG}
If $g(a,b,\epsilon)=\varphi(b)G(a,b,\epsilon)$
for some smooth function $\varphi$ and $G$
with $\varphi(0)\ne 0$,
then $\lambda$ defined in \eqref{def_lambda}
is equal to \begin{equation}\label{lambda_simpleG}
  \lambda(s)
  = \ln\frac{f(a_\alpha(s),0,0)}{f(a_\omega(s),0,0)}
  + \int_{\gamma(s)}\frac{\partial_ah(a,b,0)}{h(a,b,0)}\;da
  + \int_{\gamma(s)}\frac{\partial_bG(a,b,0)}{G(a,b,0))}\;db.
\end{equation}
\end{proposition}

\begin{proof}
From the condition $g(a,b,\epsilon)=\varphi(b)G(a,b,\epsilon)$,
\begin{align*}
  &\int_{\gamma(s)}\frac{\partial_bg(a,b,0)}{g(a,b,0)}\;db
  =\int_{\gamma(s)}\frac{\partial_b[\varphi(b)G(a,b,0)]}{\varphi(b)G(a,b,0)}\;db
  \\
  &\qquad
  =\int_{\gamma(s)}\frac{\varphi'(b)G(a,b,0)+ \varphi(b)\partial_bG(a,b,0)}{\varphi(b)G(a,b,0)}\;db
  \\
  &\qquad
  =\int_{\gamma(s)}\frac{\varphi'(b)}{\varphi(b)}\;db
  +\int_{\gamma(s)}\frac{\partial_bG(a,b,0)}{G(a,b,0)}\;db
\end{align*}
Note that the $b$-coordinates of the endpoints of $\gamma(s)$ are $0$.
Since $\varphi(0)\ne 0$,
$\int_{\gamma(s)}\frac{\varphi'(b)}{\varphi(b)}\;db
=\ln \varphi(b)\big|_{b=0}^0=0$\,
and consequently
\begin{equation}\notag
  \int_{\gamma(s)}\frac{\partial_bg(a,b,0)}{g(a,b,0)}\;db
  =\int_{\gamma(s)}\frac{\partial_bG(a,b,0)}{G(a,b,0)}\;db.
\end{equation}
Therefore \eqref{def_lambda} yields \eqref{lambda_simpleG}.
\end{proof}

To prove Theorem~\ref{thm_general},
we recall the following two theorems from \cite{Hsu:2019}.

\begin{theorem}[Theorem 5.1 in \cite{Hsu:2019}]\label{thm_entryexit2}
Consider system \eqref{deq_ab},
where $f$, $g$ and $h$ are $C^{r+1}$ functions, $r\in \mathbb N$,
that satisfy \eqref{turning_f}--\eqref{turning_g}.
Assume that $a_0<0<a_1$ satisfies relation
\begin{equation}\label{entryexit_ab}
  \int_{a_1}^{a_0} \frac{g(a,0,0)}{f(a,0,0)}\;da=0
\end{equation}
and that there exist trajectories $\gamma_1$ and $\gamma_2$ of the limiting system \begin{equation}\label{fast_ab2}
  \dot{a}= b\, h(a,b,0),
  \quad
  \dot{b}= b\,g(a,b,0),
\end{equation} such that $(a_1,0)$ is the omega-limit point of $\gamma_1$
and $(a_2,0)$ is the alpha-limit point of $\gamma_2$.
Then for all sufficiently small $\delta_1>0$ and $\delta_2>0$,
there exists $\epsilon_0>0$ such that the following holds.
Let \begin{equation}\label{def_ain_ab}
  (a^\IN,\delta_1)
  = \gamma_1\cap \{b=\delta_1\}
  \quad\text{and}\quad
  (a^\OUT,\delta_1)
  = \gamma_2\cap \{b=\delta_1\}.
\end{equation} Let \begin{equation}\label{def_Sigma_ab}
  \Sigma^\IN
  = \{(a,\delta_1): |a-a^\IN|<\delta_2\}
  \quad\text{and}\quad
  \Sigma^\OUT
  = \{(a,\delta_1): |a-a^\OUT|<\frac{|a_1|}{2}\}.
\end{equation} 
Then the transition mapping from $\Sigma^\IN$ to $\Sigma^\OUT$ of \eqref{deq_ab}
is well-defined for $\epsilon\in (0,\epsilon_0]$,
and is $C^r$ up to $\epsilon=0$.
That is, there exists a $C^r$ function \begin{equation}\notag
  \pi_\epsilon(z): \Sigma^\IN\times [0,\epsilon_0]\to \Sigma^\OUT
\end{equation}
such that, for each $z\in \Sigma^\IN$ and $\epsilon\in (0,\epsilon_0]$,
$z$ and $\pi_\epsilon(z)$ are connected
by a trajectory of \eqref{deq_ab}, and \begin{equation}\notag
  \pi_0(z_0)= \pi_0(z_1)
  \quad
  \text{for $z_0=(a_0,\delta)$ and $z_1=(a_1,\delta)$ satisfies \eqref{entryexit_ab}.}
\end{equation}
The time span $T_{\epsilon,\delta_1}$
of the trajectory $\sigma_\epsilon$
connecting $z\in \Sigma^\IN$ and $\pi_\epsilon(z)$
satisfies \begin{equation}\label{est_Teps_ab2}
  T_{\epsilon,\delta_1}=\frac{1}{\epsilon}\left(
    \int_{a_0}^{a_1}\frac{1}{f(a,0,0)}\;da
    + o(1)
  \right)\quad \text{as }\epsilon\to 0.
\end{equation} Moreover, there exists $M>0$ such that for each $\Delta\in (0,\delta_1]$,
if we parameterize $\sigma_\epsilon\cap \{b<\Delta\}$
by $(a_\epsilon(t),b_\epsilon(t))$, $t\in [0,T_{\epsilon,\Delta}]$,
then there exists $\epsilon_\Delta>0$ satisfying
\begin{equation}\label{ineq_bint_Delta}
  \int_0^{T_{\epsilon,\Delta}}
  b_\epsilon(t)\; dt
  \le M \Delta
  \quad\forall\; \epsilon\in (0,\epsilon_\Delta].
\end{equation}
\end{theorem}

The next theorem
is the variation of Floquet Theory.

\begin{theorem}[Theorem 5.2 in \cite{Hsu:2019}]
\label{thm_floquet}
Consider systems in $\mathbb R^N$, $N\ge 2$, 
of the form \begin{equation}\label{deq_z}
  \dot{z}= {h}_\epsilon({z}),
\end{equation}
where ${h}_\epsilon(z)=h(z,\epsilon)$
is a $C^2$ function of $({z},\epsilon)\in \mathbb R^N\times [0,\epsilon_0]$.
Let $z_0\in \mathbb R^N$ with $h(z_0,0)\ne 0$,
and let $\Sigma$ be a cross section of $z_0$
transversal to $h_0(z_0)$.
Assume that there exist $\epsilon_0>0$
and a neighborhood $\Sigma_{(1)}\subset \Sigma$ of $z_0$
such that the return map from $\Sigma_{(1)}$ to $\Sigma$
is well-defined for $\epsilon\in (0,\epsilon_0]$
and is $C^1$ up to $\epsilon=0$.
That is, there is a $C^1$ function
\begin{equation}\notag
  P_\epsilon(z): 
  \Sigma_{(1)}\times [0,\epsilon_0]\to \Sigma
\end{equation}
such that
for any $z\in \Sigma_{(1)}$ and $\epsilon\in (0,\epsilon_0]$
there is a trajectory of \eqref{deq_z}
that starts at $z\in \Sigma_{(1)}$
and returns to $\Sigma$ at $P_\epsilon(z)$.

Let $\zeta_\epsilon(t)$, $0\le t\le T_\epsilon$,
be a trajectory of \eqref{deq_z} that starts at $z_0$
and ends at $P_\epsilon(z_0)$.
Assume that \begin{equation}\notag
  \int_0^{T_\epsilon}
  \mathrm{div}(h_\epsilon(\zeta_\epsilon(t)))\;dt
  \to \lambda_0
  \quad\text{as }\epsilon\to 0
\end{equation}
for some $\lambda_0\in \mathbb R$.
Then \begin{equation}\label{eq_detDP}
  \det\big(DP_{0}(u_0)\big)
  = \exp(\lambda_0),
\end{equation}
where $DP_{\epsilon}(u_0)$ is regarded
as a linear transform on the tangent space $T_{z_0}\Sigma$.
\end{theorem}

Now we use Theorems~\ref{thm_entryexit2} and~\ref{thm_floquet}
to prove Theorem~\ref{thm_general}.

\begin{proof}[Proof of Theorem \ref*{thm_general}]
The proof for the case with $\chi(s_0)\ne 0$
is similar to
that in \cite[Theorem 2.1]{Hsu:2019},
so we omit it here.
In this case, there is a neighborhood $U$ of $z_1$
such that no periodic orbit of \eqref{deq_ab}
intersects $U$ for any sufficiently small $\epsilon>0$,

Assume $\chi(s_0)=0$ and $\lambda(s_0)\ne 0$.
Let $a_0=a_\omega(s_0)$ and $a_1=a_\alpha(s_0)$.
The condition $\chi(a_0)=0$ means that \eqref{entryexit_ab} holds.
Let $(a^\IN,\delta_1)$, $(a^\OUT,\delta_1)$,
$\Sigma^\IN$ and $\Sigma^\OUT$
be the points and segments defined in \eqref{def_ain_ab} and \eqref{def_Sigma_ab}.
By Theorem~\ref{thm_entryexit2},
The transition map $\pi^{(1)}_\epsilon:\Sigma^\IN\to \Sigma^{\OUT}$, $\epsilon\in (0,\epsilon_0]$
is well defined and is $C^1$ up to $\epsilon=0$.
Let \begin{equation}\notag
  \widetilde{\Sigma}^{\IN}
  =\left\{(a,\delta_1): a\in (-\infty,\bar{a})\right\}.
\end{equation}
Since \eqref{deq_ab} is a regular perturbation of \eqref{fast_ab}
in the region $\{(a,b): b\ge \delta_1\}$,
the transition map $\pi^{(2)}_\epsilon:\Sigma^{\OUT}\to \widetilde{\Sigma}^{\IN}$
for system \eqref{deq_ab}
is well-defined and is smooth for $\epsilon\in [0,\epsilon_0]$.
Since $(a^\IN,\delta),(a^\OUT,\delta)\in \gamma(s_0)$,
we have $\pi_0^{(2)}(a^\OUT,\delta_1)= (a^\IN,\delta_1)$.

Let $P_\epsilon=\pi_\epsilon^{(2)}\circ \pi_\epsilon^{(1)}$.
Then $P_0(a^\IN,\delta_1)= \pi_0^{(2)}(a^\OUT,\delta_1)= (a^\IN,\delta_1)$.
That is, $(a^\IN,\delta_1)$ is a fixed point of $P_0$.
To show that this fixed point persists for small $\epsilon>0$,
by the implicit function theorem it suffices to show that
the Jacobian matrix of the return map $P_0$ evaluated at $(a^\IN,\delta)$
is non-singular.
Let \begin{equation}\label{def_mu_ab}
  \mu_\epsilon
  \equiv \int_{\zeta_\epsilon}
  \mathrm{div}\begin{pmatrix}
    \epsilon f(a,b,\epsilon)+ b\,h(a,b,\epsilon)\\ b\, g(a,b,\epsilon)
  \end{pmatrix}\;dt,
\end{equation}
where $\zeta_\epsilon$ is the trajectory of \eqref{deq_ab}
that starts at $(a^\IN,\delta)$ and ends at $P_\epsilon(a^\IN,\delta)$.
By Theorem~\ref{thm_floquet},
it suffices to show that $\mu_\epsilon$
approaches a nonzero value as $\epsilon\to 0$.

Using the equation for $\dot{a}$ in \eqref{deq_ab},
\begin{equation}\label{zeta_af}
  \int_{\zeta_\epsilon} \partial_a \big(\epsilon f(a,b,\epsilon)+b\,h(a,b,\epsilon)\big)\;dt
  =\int_{\zeta_\epsilon} \frac{\partial_a \big(\epsilon f(a,b,\epsilon)+b\,h(a,b,\epsilon)\big)}
  {\epsilon f(a,b,\epsilon)+b\,h(a,b,\epsilon)}\;da.
\end{equation}
For any fixed $\Delta\in (0,\delta_1)$,
since $b\,h(a,b,0)$ is bounded away from zero
on ${\zeta_\epsilon\cap \{b>\Delta\}}$,
by regular perturbation theory we have \begin{equation}\notag
  \lim_{\epsilon\to 0}
  \int_{\zeta_\epsilon\cap \{b>\Delta\}}
  \frac{\partial_a \big(\epsilon f(a,b,\epsilon)+b\,h(a,b,\epsilon)\big)}
  {\epsilon f(a,b,\epsilon)+b\,h(a,b,\epsilon)}\;da
  = \int_{\gamma(s_0)\cap \{b> \Delta\}}
  \frac{\partial_ah(a,b,0)}{h(a,b,0)}\;da.
\end{equation}
Hence
\begin{equation}\label{zeta_est1a}
  \lim_{\epsilon\to 0}
  \left|
  \int_{\zeta_\epsilon\cap \{b>\Delta\}}
  \frac{\partial_a \big(\epsilon f(a,b,\epsilon)+b\,h(a,b,\epsilon)\big)}
  {\epsilon f(a,b,\epsilon)+b\,h(a,b,\epsilon)}\;da
  - \int_{\gamma(s_0)}
  \frac{\partial_ah(a,b,0)}{h(a,b,0)}\;da
  \right|
  \le C\Delta.
\end{equation}
Here and in the rest of the proof we use $C$
to denote constants independent of $\epsilon$ and $\Delta$.
Next we claim that \begin{equation}\label{est_f_Delta}
  \limsup_{\epsilon\to 0}
  \left|
  \int_{\zeta_\epsilon\cap \{b<\Delta\}}
  \frac{\epsilon\,\partial_af(a,b,\epsilon)}
  {\epsilon f(a,b,\epsilon)+b\,h(a,b,\epsilon)}\;da
  - \ln\frac{f(a_1,0,0)}{f(a_0,0,0)}
  \right|
  \le C\Delta
\end{equation}
and that
\begin{equation}\label{est_h_Delta}
  \limsup_{\epsilon\to 0}
  \left|
    \int_{\zeta_\epsilon\cap \{b<\Delta\}}
    \frac{b\,\partial_a h(a,b,\epsilon)}
    {\epsilon f(a,b,\epsilon)+b\,h(a,b,\epsilon)}\;da
  \right|
  \le C\Delta.
\end{equation} 
We write
$\zeta_\epsilon\cap \{b<\Delta\}
=\zeta_{\epsilon,\Delta}^{(1)}
\cup \zeta_{\epsilon,\Delta}^{(2)}
\cup \zeta_{\epsilon,\Delta}^{(3)}$
by \begin{align}
  &\zeta_{\epsilon,\Delta}^{(1)}
  =\zeta_\epsilon\cap \{b<\Delta\}\cap \{a<a_0+\Delta\},
  \notag
  \\
  &\zeta_{\epsilon,\Delta}^{(2)}
  =\zeta_\epsilon\cap \{b<\Delta\}\cap \{a_0+\Delta<a<a_1+\Delta\},
  \notag
  \\
  &\zeta_{\epsilon,\Delta}^{(3)}
  =\zeta_\epsilon\cap \{b<\Delta\}\cap \{a>a_1-\Delta\}.
  \notag
\end{align}
It can be shown (from the proof of \cite[Theorem 2.1]{Hsu:2019})
that, for some $K=K(\Delta)>0$
independent of $\epsilon$, \begin{equation}\label{est_b_Delta}
  b<e^{-K/\epsilon}
  \quad\text{on}\;\;
  \zeta_{\epsilon,\Delta}^{(2)}.
\end{equation}
Since $f(a,0,0)$ is bounded away from zero,
\eqref{est_b_Delta} gives \begin{align*}
  &\lim_{\epsilon\to 0}
  \int_{\zeta_{\epsilon,\Delta}^{(2)}}
  \frac{\epsilon\, \partial_af(a,b,\epsilon)}
  {\epsilon f(a,b,\epsilon)+b\,h(a,b,\epsilon)}\;da
  =
  \lim_{\epsilon\to 0}\int_{\zeta_{\epsilon,\Delta}^{(2)}}
  \frac{\partial_af(a,b,\epsilon)}
  {f(a,b,\epsilon)+O(e^{-K/\epsilon}/\epsilon)}\;da
  \\
  &\qquad
  =\int_{a_0+\Delta}^{a_1-\Delta}
  \frac{\partial_af(a,0,0)}
  {f(a,0,0)}\;da
  = \ln\frac{f(a_1-\Delta,0,0)}{f(a_0+\Delta,0,0)}.
\end{align*}
Hence \begin{equation}\label{est_f_zeta2}
  \lim_{\epsilon\to 0}
  \left|
  \int_{\zeta_{\epsilon,\Delta}^{(2)}}
  \frac{\epsilon\,\partial_af(a,b,\epsilon)}
  {\epsilon f(a,b,\epsilon)+b\,h(a,b,\epsilon)}\;da
  -
  \ln\frac{f(a_1,0,0)}{f(a_0,0,0)}
  \right|
  \le C\Delta.
\end{equation}
On the other hand,
since $g(a_0,0,0)<0$
and $\zeta_{\epsilon,\Delta}^{(1)}$
lies in a neighborhood of $(a_0,0)$,
for any fixed positive number $\beta< |g(a_0,0,0)|$
we have
$b(t)\le Ce^{-\beta t}$ on $\zeta_{\epsilon,\Delta}^{(1)}$
if $(a(0),b(0))$ is the entry point of $\zeta_{\epsilon,\Delta}^{(1)}$.
Since $f(a,0,0)$ is bounded away from zero,
the equation $\dot{a}=\epsilon f+b\,h$
and the estimate $b(t)\le Ce^{-\beta t}$
imply that the time span $T_\epsilon^{(1)}$
of $\zeta_{\epsilon,\Delta}^{(1)}$
satisfies $T_\epsilon^{(1)}\le C\Delta/\epsilon$.
Therefore
\begin{equation}\label{est_f_zeta1}
  \left|
    \int_{\zeta_{\epsilon,\Delta}^{(1)}}
    \frac{\epsilon\,\partial_af(a,b,\epsilon)}
    {\epsilon f(a,b,\epsilon)+b\,h(a,b,\epsilon)}\;da
  \right|
  = \left|
    \int_0^{T_\epsilon^{(1)}}
    \epsilon\,\partial_af(a,b,\epsilon)
    \;dt
  \right|
  \le C\Delta.
\end{equation}
Similarly,
\begin{equation}\label{est_f_zeta3}
  \left|
    \int_{\zeta_{\epsilon,\Delta}^{(3)}}
    \frac{\epsilon\,\partial_af(a,b,\epsilon)}
    {\epsilon f(a,b,\epsilon)+b\,h(a,b,\epsilon)}\;da
  \right|
  \le C\Delta.
\end{equation}
Combining estimates
\eqref{est_f_zeta2}, \eqref{est_f_zeta1} and \eqref{est_f_zeta3},
we obtain \eqref{est_f_Delta}.

Since $f(a,0,0)$ is bounded away from zero,
\eqref{est_b_Delta} gives \begin{equation}\label{est_h_zeta2}
  \int_{\zeta_{\epsilon,\Delta}^{(2)}}
  \frac{b\; \partial_ah(a,b,\epsilon)}
  {\epsilon f(a,b,\epsilon)+b\,h(a,b,\epsilon)}\;da
  =
  \int_{\zeta_{\epsilon,\Delta}^{(2)}}
  \frac{O(e^{-K/\epsilon}/\epsilon)}
  {f(a,b,\epsilon)+O(e^{-K/\epsilon}/\epsilon)}\;da
  \to 0
\end{equation}
as $\epsilon\to 0$.
On the other hand,
by the equations for $\dot{a}$ and $\dot{b}$ in \eqref{fast_ab},
\begin{equation}\notag
  \int_{\zeta_{\epsilon,\Delta}^{(1)}}
  \frac{b\,\partial_a h(a,b,\epsilon)}
  {\epsilon f(a,b,\epsilon)+b\,h(a,b,\epsilon)}\;da
  =\int_{\zeta_{\epsilon,\Delta}^{(1)}}
  \frac{b\,\partial_a h(a,b,\epsilon)}
  {b\,g(a,b,\epsilon)}\;db
  =\int_{\zeta_{\epsilon,\Delta}^{(1)}}
  \frac{\partial_a h(a,b,\epsilon)}
  {g(a,b,\epsilon)}\;db.
\end{equation}
Since $g(a_0,0,0)\ne 0$
and $\mathrm{length}\big(\zeta_{\epsilon,\Delta}^{(1)}\big)
\le C\Delta$,
it follows that
\begin{equation}\label{est_h_zeta1}
  \left|
  \int_{\zeta_{\epsilon,\Delta}^{(1)}}
  \frac{b\,\partial_a h(a,b,\epsilon)}
  {\epsilon f(a,b,\epsilon)+b\,h(a,b,\epsilon)}\;da
  \right|
  = \left|
    \int_{\zeta_{\epsilon,\Delta}^{(1)}}
    \frac{\partial_a h(a,b,\epsilon)}
    {g(a,b,\epsilon)}\;db
  \right|
  \le C\Delta.
\end{equation}
Similarly, from $g(a_1,0,0)\ne 0$, \begin{equation}\label{est_h_zeta3}
  \left|
  \int_{\zeta_{\epsilon,\Delta}^{(3)}}
  \frac{b\,\partial_a h(a,b,\epsilon)}
  {\epsilon f(a,b,\epsilon)+b\,h(a,b,\epsilon)}\;da
  \right|
  \le C\Delta.
\end{equation}
Combining \eqref{est_h_zeta2}, \eqref{est_h_zeta1} and \eqref{est_h_zeta3},
we obtain \eqref{est_h_Delta}.

By \eqref{zeta_est1a}, \eqref{est_f_Delta} and \eqref{est_h_Delta},
\begin{equation}\notag
  \limsup_{\epsilon\to 0}
  \left|
  \int_{\zeta_\epsilon} \partial_a 
  \big(\epsilon f+b\,h\big)\;dt
  -\ln\frac{f(a_1,0,0)}{f(a_0,0,0)}
  - \int_{\gamma(s_0)}\frac{\partial_ah(a,b,0)}{h(a,b,0)}\;da
  \right|
  \le C\Delta,
\end{equation}
where $(\epsilon f+b\,h)$
is evaluated at $(a,b,\epsilon)$.
Since $\Delta$ can be arbitrarily small,
we obtain
\begin{equation}\label{zeta_est1}
  \lim_{\epsilon\to 0}
  \int_{\zeta_\epsilon} \partial_a
  \big(\epsilon f(a,b,\epsilon)+b\,h(a,b,\epsilon)\big)\;dt
  =\ln\frac{f(a_1,0,0)}{f(a_0,0,0)}
  + \int_{\gamma(s_0)}\frac{\partial_ah(a,b,0)}{h(a,b,0)}\;da.
\end{equation}

A similar calculation gives \begin{equation}\label{zeta_est2}
  \lim_{\epsilon\to 0}
  \int_{\zeta_\epsilon} \partial_b \big(b\,g(a,b,\epsilon)\big)\;dt
  = \int_{\gamma(s_0)} 
  \frac{\partial_bg(a,b,0)}
  {h(a,b,0)}\;da.
\end{equation}
By \eqref{zeta_est1} and \eqref{zeta_est2} we conclude that
the number $\mu_\epsilon$ defined by \eqref{def_mu_ab} satisfies
\begin{equation}\notag
  \lim_{\epsilon\to 0}
  \mu_\epsilon
  = \ln\frac{f(a_1,0,0)}{f(a_0,0,0)}
  + \int_{\gamma(s_0)}\frac{\partial_ah(a,b,0)+\partial_bg(a,b,0)}{h(a,b,0)}\;da.
\end{equation}
Using the relation $db/da=g(a,b,0)/h(a,b,0)$ from \eqref{fast_ab},
it follows that $\lim_{\epsilon\to 0}\mu_\epsilon=\lambda(s_0)$
with $\lambda$ defined by \eqref{def_lambda}.

By assumption, $\lambda(s_0)\ne 0$,
by Theorem~\ref{thm_floquet}
the return map $P_\epsilon$ has a unique fixed point near $(a^\IN,\delta_1)$
for all small $\epsilon>0$,
and $P_\epsilon$ is a contraction if $\lambda(s_0)<0$,
and an expansion if $\lambda(s_0)>0$.
Hence there is a unique periodic orbit
$\ell_\epsilon$
of \eqref{deq_ab}
near $\gamma(s_0)$
for every small $\epsilon>0$.
This periodic orbit is locally orbitally asymptotically stable if $\lambda(s_0)<0$,
and is unstable if $\lambda(s_0)>0$.
The estimate \eqref{est_Teps_ab} follows from \eqref{est_Teps_ab2}.
\end{proof}

\section{The Chemostat Predator-Prey System}
\label{sec_predprey}

The restriction of system \eqref{deq_predprey} on
the invariant plane $\Lambda$ is
governed by $S=S^0-\rho x-c\rho y$
and
\begin{equation}\label{deq_predprey_xy}\begin{aligned}
  \dot{x}
  &= x\big(-\epsilon+ m(S^0-\rho x-c\rho y)\big)- cyp(x)
  \\
  &\equiv
  c\left(\rho m x+p(x)\right)\,\big(F_\epsilon(x)-y\big),
  \\
  \dot{y}
  &= y\big(-\epsilon+p(x)\big),
\end{aligned}\end{equation}
where \begin{equation}\label{def_Feps}
  F_\epsilon(x)
  = \frac{x\left(-\epsilon+mS^0-\rho mx\right)}
  {c\left(\rho mx+p(x)\right)}.
\end{equation}
Define $F(x)=F_0(x)$, that is \begin{equation}\label{predprey_F0}
  F(x)= \frac{x(mS^0-\rho mx)}{c\left(\rho mx+p(x)\right)},
\end{equation}
and define \begin{equation}\notag
  \bar{y}=F(0)=\lim_{x\to 0}F(x)=\frac{mS^0}{c(\rho m+p'(0))}>0.
\end{equation}
The last inequality follows from $p'(0)> 0$ in condition~\eqref{cond_p}.
Note that $F_\epsilon(x)$
is continuous at $(\epsilon,x)=(0,0)$
by setting $F_\epsilon(0)=\frac{-\epsilon+mS^0}{c(\rho m+p'(0))}$.
Also note that $F(x)>0$ for all $x\in (0,S^0/\rho)$,
and that $F(S^0/\rho)=0$.

When $\epsilon=0$, system \eqref{deq_predprey_xy} 
reduces to \begin{equation}\label{fast_predprey}\begin{aligned}
  &\dot{x}= c\left(\rho m x+p(x)\right)\,\big(F(x)-y\big),
  \\
  &\dot{y}= yp(x).
\end{aligned}\end{equation}
Since $p(0)=0$ and $F(S^0/\rho)=0$,
system \eqref{fast_predprey}
has a line of equilibria $\{x=0\}$
and an isolated saddle equilibrium $E_1=(S^0/\rho,0)$.
The unstable manifold of $E_1$
is the portion of the line $\{x+cy=S^0/\rho\}$.
Given any $x_0\in (0,S^0/\rho)$,
since $p(x)>0$ for all $x\in (0,S^0/\rho)$,
both the forward and backward trajectories
starting at $(x,y)=(x_0,F(x_0))$
approach points on the
boundary of $\Lambda$
(see Figure~\ref{fig_predprey_H2gamma}).
This gives a continuous family of heteroclinic orbits
parameterized by $x_0\in (0,S^0/\rho)$
that fills the region
in the positive quadrant
bounded below by the unstable manifold of $E_1$.

Denote the trajectory passing through $(x_0,F(x_0))$ by $\gamma(x_0)$,
and the $y$-values
at the omega- and alpha-limit points of $\gamma(x_0)$
by $y_\omega(x_0)$ and $y_\alpha(x_0)$, respectively.
Denote the segment 
connecting the two endpoints of $\gamma(x_0)$ by $\sigma(x_0)$.
Then $\gamma(x_0)\cup \sigma(x_0)$
is a singular closed orbit for each $x_0\in (0,S^0/\rho)$

Applying Theorem~\ref{thm_general} to system \eqref{deq_predprey_xy}
with $(-y,x)$ playing the role of $(a,b)$,
the formula \eqref{def_chi} gives,
up to multiplication by positive constants,
\begin{equation}\label{chi_predprey}
  \chi(x_0)
  = \int_{y_\alpha(x_0)}^{y_\omega(x_0)}
  \frac{y-\bar{y}}{y}\;dy.
\end{equation}

By Propositions~\ref{prop_lambdaFH} and~\ref{prop_lambdaH},
the formula \eqref{def_lambda} gives, 
up to multiplication by positive constants,
\begin{equation}\label{lambda_predprey}
  \lambda(x_0)
  =\int_{\gamma(x_0)}
  \frac{\partial_x\big[F(x)-y\big]}
  {F(x)-y}\;dy
  =\int_{\gamma(x_0)}
  \frac{F'(x)}{F(x)-y}\;dy.
\end{equation}

\begin{theorem}
\label{thm_predprey}
Assume that $x_0\in (0,S^0/\rho)$ satisfies $\chi(x_0)=0$ and $\lambda(x_0)\ne 0$,
where $\chi$ and $\lambda$ are defined in \eqref{chi_predprey} and \eqref{lambda_predprey}.
Then for any sufficiently small $\epsilon>0$,
there is a periodic orbit ${\ell}_\epsilon$ of \eqref{deq_predprey}
in a $O(\epsilon)$-neighborhood of $\Gamma(x_0)=\gamma(x_0)\cup \sigma(x_0)$.
The minimal period of $\ell_\epsilon$, denoted by $T_\epsilon$, satisfies \begin{equation}\label{est_Teps}
  T_\epsilon= \frac{1}{\epsilon}\Big(
    \ln\left(\frac{y_\omega(x_0)}{y_\alpha(x_0)}\right)+o(1)
  \Big)
  \quad\text{as }\epsilon\to 0.
\end{equation}
Moreover, 
${\ell}_\epsilon$ is locally orbitally asymptotically stable if $\lambda(x_0)<0$,
and is orbitally unstable if $\lambda(x_0)>0$.
Conversely, 
if $\chi(x_0)\ne 0$,
then for any point $z_1$ in the interior of the trajectory $\gamma(x_0)$,
there is a neighborhood $U$ of $z_1$
such that no periodic orbit of \eqref{deq_predprey}
intersects $U$ for any sufficiently small $\epsilon>0$.
\end{theorem}

\begin{proof}
If $\chi(x_0)= 0$ and $\lambda(x_0)\ne 0$,
then by Theorem~\ref{thm_general},
for every small $\epsilon>0$,
system \eqref{deq_predprey_xy} has unique periodic orbit $\ell_\epsilon$
near $\gamma(x_0)\subset \Lambda$.
Since $\Lambda$ is invariant under \eqref{deq_predprey_xy},
$\ell_\epsilon$ is also a periodic orbit  for \eqref{deq_predprey}.
If $\lambda(x_0)>0$, then $\ell_\epsilon$
is orbitally unstable for \eqref{deq_predprey_xy},
and therefore $\ell_\epsilon$ is orbitally unstable for \eqref{deq_predprey}.
If $\lambda(x_0)<0$, then $\ell_\epsilon$
is locally orbitally asymptotically stable for \eqref{deq_predprey_xy}.
Since $\Lambda$ is a hyperbolic attractor,
it follows that $\ell_\epsilon$ is locally orbitally asymptotically stable for \eqref{deq_predprey}.

On the other hand,
If $\chi(x_0)\ne 0$ and $\lambda(x_0)\ne 0$,
then by Theorem~\ref{thm_general},
for every small $\epsilon>0$,
there is no periodic orbit of \eqref{deq_predprey_xy} near $\gamma(x_0)\subset \Lambda$.
Since $\Lambda$ is a global attractor,
it follows that
there is no periodic orbit of \eqref{deq_predprey} near $\gamma(x_0)$.
\end{proof}

The function $\chi(x_0)$ defined by  \eqref{chi_predprey}
can be expressed as
a line integral along the trajectory $\gamma(x_0)$ as follows.

\begin{proposition}
\label{prop_chi_x}
The the function $\chi$ defined by \eqref{chi_predprey} satisfies
\begin{equation}\label{predprey_chi_x}
  \chi(x_0)
  =\int_{\gamma(x_0)}
  \frac{p(x)}{\rho mx+p(x)}\;
  \frac{F(x)-F(0)}{F(x)-y}\;dx.
\end{equation}
\end{proposition}

\begin{proof}
Fix any $x_0\in (0,S^0/\rho)$.
Let $(x(t),y(t))$
be the solution of \eqref{fast_predprey} 
with trajectory $\gamma(x_0)$.
Recall that $(0,y_\alpha(x_0))$
and $(0,y_\omega(x_0))$
are the alpha- and omega-limit point, respectively,
of the trajectory $\gamma(x_0)$.
From the equation for $\dot{y}$ in \eqref{fast_predprey},
equation \eqref{chi_predprey} can be written as
\begin{equation}\label{predprey_chi_t}
  \chi(x_0)= \int_{-\infty}^{\infty}(y-\bar{y})\,p(x)\;dt,
\end{equation}
where $\bar{y}=F(0)$.
On the other hand,
from the equation for $\dot{x}$ in \eqref{fast_predprey},
\begin{equation}\label{predprey_F_t}
  \int_{-\infty}^{\infty} (\rho mx+p(x))\,(F(x)-y)\;dt=0.
\end{equation}
From \eqref{predprey_chi_t} and \eqref{predprey_F_t},
it follows that \begin{equation}\label{predprey_F_F}
  \chi(x_0)= \int_{-\infty}^{\infty}p(x)(F(x)-\bar{y})\;dt
  +\int_{-\infty}^{\infty}\rho m x(F(x)-y)\;dt.
\end{equation}
Note that \begin{align*}
  \int_{-\infty}^{\infty}\rho m x(F(x)-y)\;dt
  &=\int_{-\infty}^{\infty} \frac{\rho m x}{\rho m x+ p(x)} x'(t)\;dt
  \\
  &=\eta(x(t))\big|_{t=-\infty}^{\infty}
\end{align*}
where $\eta(u)=\int_0^u \frac{\rho m x}{\rho m x+ p(x)}\;dx$,
which is continuous at $u=0$
because $p(0)=0$ and $p'(0)> 0$.
Since $x(\infty)=x(-\infty)=0$,
it follows that \begin{equation}\label{predprey_F_0}
  \int_{-\infty}^{\infty}\rho m x(F(x)-y)\;dt=0.
\end{equation}
Equations \eqref{predprey_F_F} and \eqref{predprey_F_0} give
\begin{equation}\notag
  \chi(x_0)
  = \int_{-\infty}^{\infty}p(x)(F(x)-\bar{y})\;dt.
\end{equation}
Therefore,
from the equation for $\dot{x}$ in \eqref{fast_predprey},
equation \eqref{predprey_chi_x} follows.
\end{proof}

Next we assume that the function $F(x)$ in \eqref{deq_predprey}
satisfies the one-hump condition:
For some $\widehat{x}\in (0,S^0/\rho)$,
\begin{equation}\label{cond_1hump}
  F'(x)>0\quad \forall\; x\in (0,\widehat{x}),
  \quad\text{and}\quad
  F'(x)<0\quad \forall\; x\in (\widehat{x},S^0/\rho).
\end{equation}

The following result is similar to \cite[Theorem 3.1]{Hsu:2019}.

\begin{theorem}
\label{thm_1hump}
Assume \eqref{cond_1hump} holds for
$F(x)$ defined by \eqref{predprey_F0}.
Then system \eqref{deq_predprey}
has a globally orbitally asymptotically stable
(with respect to all positive non-stationary solutions)
periodic orbit
for all small $\epsilon>0$.
\end{theorem}

We will use the following two lemmas.

\begin{lemma}
\label{lemma_1hump}
Assume \eqref{cond_1hump}.
Then the function $\chi$ defined by \eqref{chi_predprey}
has a unique root $x_0$ in $(0,S^0/\rho)$,
and it satisfies $\lambda(x_0)<0$.
\end{lemma}

\begin{proof}
Note that condition \eqref{cond_1hump}
implies that there exists a unique value $\bar{x}\in (\widehat{s},S^0/\rho)$
such that $F(\bar{x})=F(0)$.

First we claim that \begin{equation}\label{predprey_chi_pos}
  \chi(x_0)>0
  \;\;\text{for all}\;\;
  x_0\in (0,\bar{x}].
\end{equation}
Fix any $x_0\in (0,\widehat{x})$.
We parameterize $\gamma(x_0)$ by \begin{equation}\label{def_Ypm}
  \gamma(x_0)
  =\{(x,Y_-(x)): x\in (0,x_0]\}
  \cup \{(x,Y_+(x)): x\in (0,x_0]\}
\end{equation} 
with \begin{equation}\label{ineq_Ypm_F}
  Y_-(x)<F(x)
  \quad\text{and}\quad
  Y_+(x)>F(x)
  \quad\forall\; x\in (0,x_0).
\end{equation}
Then equation \eqref{predprey_chi_x} in Proposition \ref{prop_chi_x}
yields
\begin{equation}\notag
  \chi(x_0)
  = \int_0^{x_0}\frac{p(x)}{\rho mx+p(x)}
  \left[
    \frac{F(x)-F(0)}{F(x)-Y_-(x)}
    - \frac{F(x)-F(0)}{F(x)-Y_+(x)}
  \right]\;dx.
\end{equation}
Since $F(x)-F(0)>0$ for $0<x<x_0<\bar{x}$,
by \eqref{ineq_Ypm_F}
it follows that $\chi(x_0)>0$.

Next we claim that \begin{equation}\label{predprey_lambda_neg}
  \lambda(x_0)<0
  \;\;\text{for all}\;\;
  x_0\in [\bar{x},S^0/\rho).
\end{equation}
Fix any $x_0\in [\bar{x},S^0/\rho)$.
From the definition of $\lambda(x_0)$ in \eqref{lambda_predprey},
using \eqref{def_Ypm} we have \begin{equation}\label{lambda_int_F}
  \lambda(x_0)
  =\int_{0}^{x_0}F'(x)\left(
    \frac{1}{F(x)-Y_-(x)}+\frac{1}{Y_+(x)-F(x)}
  \right) dx.
\end{equation}
Since $F'(x)$ for $x\in (\bar{x},x_0]$,
by \eqref{ineq_Ypm_F}
we obtain \begin{equation}\notag
  \lambda(x_0)
  <\int_0^{\bar{x}}
  F'(x)\left(
    \frac{1}{F(x)-Y_-(x)}+\frac{1}{Y_+(x)-F(x)}
  \right) dx.
\end{equation}
Since $Y_+(x)$ is decreasing and $Y_-(x)$ is increasing,
condition \eqref{cond_1hump} yields \begin{align*}
  \lambda(x_0)
  &<\int_0^{\bar{x}}
  F'(x)\left(
    \frac{1}{F(x)-Y_-(\widehat{x})}+\frac{1}{Y_+(\widehat{x})-F(x)}
  \right) dx
  \\[.5em]
  &=\left.\ln\left(
    \frac{F(x)-Y_-((\widehat{x})}{Y_+(\widehat{x})-F(x)}
  \right)\right|_{x=0}^{\bar{x}}
  = 0.
\end{align*}
The last equality follows from the condition $F(\bar{x})=F(0)$.
Hence $\lambda(x_0)<0$.

Finally, we claim that \begin{equation}\label{predprey_chi_limit}
  \lim_{x_0\to S^0/\rho}\chi(x_0)=-\infty.
\end{equation}
Note that the expression \eqref{chi_predprey} of $\chi$
can be written  as \begin{equation}\label{predprey_chi_phi}
  \chi(x)= \psi(y_\omega(x))- \psi(y_\alpha(x)),
\end{equation}
where \begin{equation}\notag
  \psi(y)=y-\bar{y}-\ln(y/\bar{y}).
\end{equation}
Note also that the functions $y_\alpha(x)$ and $y_\omega(x)$
satisfy \begin{equation}\label{predprey_y_limits}
  \lim_{x\to K^-}y_\alpha(x)= 0
  \quad\text{and}\quad
  \lim_{x\to K^-}y_\omega(x)
  \;\;\text{exists and is finite}.
\end{equation} 
Since $\lim_{y\to 0^+}\psi(y)=\infty$, 
\eqref{predprey_chi_limit} follows from
\eqref{predprey_chi_phi} and\eqref{predprey_y_limits}.

Since $\chi(x)>0$ for $x\in (0,\bar{x})$ and
$\lim_{x\to S^0/\rho}\chi(x)=-\infty$,
the continuous function $\chi(x)$
has at least one root in $(\bar{x},S^0/\rho)$.
Suppose for contradiction that
$\chi$ has two distinct roots, say $x_0<x_1$.
By \eqref{predprey_chi_pos} and \eqref{predprey_lambda_neg},
we have $\lambda(x_0)<0$ and $\lambda(x_1)<0$.
By Theorem \ref{thm_general}
there are 
locally orbitally asymptotically stable
periodic orbits 
$\ell^{(0)}_\epsilon$ and $\ell^{(1)}_\epsilon$
near $\Gamma(x_0)$ and $\Gamma(x_1)$, respectively.
Note that $\Gamma(x_0)$ is enclosed by $\Gamma(x_1)$.
By \eqref{predprey_lambda_neg} and Theorem \ref{thm_general}
there is no unstable periodic orbit 
between $\ell^{(0)}_\epsilon$ and $\ell^{(1)}_\epsilon$.
Also note that
no equilibrium lies between $\ell^{(0)}_\epsilon$ and $\ell^{(1)}_\epsilon$.
This contradicts the Poincar\'{e}-Bendixson Theorem.
Therefore $\chi$ has exactly one root $x_0$ in $(0,S^0/\rho)$.
By \eqref{predprey_chi_pos} and \eqref{predprey_lambda_neg},
this root satisfies $\bar{x}<x_0<S^0/\rho$,
and therefore $\lambda(x_0)<0$.
\end{proof}

The next lemma was derived by Wolkowicz \cite{Wolkowicz:1988},
and we omit its proof here.

\begin{lemma}
\label{lem_xfar}
If $F'(x)>0$ on $(0,\widehat{x}]$ 
or $F'(x)<0$ on $(0,\widehat{x}]$ 
for some $\widehat{x}>0$,
then
no periodic orbit of \eqref{deq_predprey_xy}
lies entirely in the strip $\{(x,y): 0<x<\widehat{x}\}$
for any sufficiently small $\epsilon>0$.
\end{lemma}

\begin{proof}[Proof of Theorem~\ref{thm_1hump}]
By Lemma~\ref{lemma_1hump},
the function $\chi$
has a unique root $x_0$ in the interval $(0,S^0/\rho)$,
and $\lambda(x_0)<0$.
From Theorem~\ref{thm_general} and Lemma~\ref{lem_xfar},
it follows that
system \eqref{deq_predprey_xy}
has a unique periodic orbit $\ell_\epsilon$ in $\Lambda$
for all small $\epsilon>0$.
It can be shown by the Butler-McGehee Lemma \cite{Butler:1986}
that the flow \eqref{deq_predprey_xy} is persistent
in the sense that
the omega-limit set of any point in $\Lambda$
does not intersect the boundary of $\Lambda$.
Therefore, by the Poincar\'{e}-Bendixon Theorem,
the periodic orbit $\ell_\epsilon$
attracts all non-stationary points in $\Lambda$.

Since $\ell_\epsilon$
is locally orbitally asymptotically stable in $\Lambda$ for \eqref{deq_predprey_xy},
by \eqref{decay_predprey}
it follows that $\ell_\epsilon$
is locally orbitally asymptotically stable in $\mathbb R^3$ for \eqref{deq_predprey}.
Given any solution $({S}(t),{x}(t),{y}(t))$ of \eqref{deq_predprey}
with a positive initial value,
it can be shown by the Butler-McGehee Lemma that
the flow \eqref{deq_predprey} is persistent
in the sense that
the omega-limit set of any point in $\mathbb R^3_+$
does not intersect the boundary of $\Lambda$.
Let $\Omega$ be the omega-limit set of $({S}(t),{x}(t),{y}(t))$.
Then $\Omega\subset \Lambda$ since $\Lambda$ is a global attractor
and solutions of \eqref{deq_predprey} are persistent.
By the positive invariance of omega-limit sets,
$\Omega\setminus \ell_\epsilon=\emptyset$.
Hence $\Omega=\ell_\epsilon$.
This implies that the trajectory converges to $\ell_\epsilon$.
\end{proof}

\begin{example}
\label{ex_H2}
Consider the Holling type II functional response,
\begin{equation}\label{def_H2}
  p(x)=\frac{bx}{a+x}.
\end{equation}
It was shown by
Bolger et al.\ \cite{Bolger:2018}
that the function $F_\epsilon(x)$ defined by \eqref{def_Feps} is concave-down.
Hence condition \eqref{cond_1hump} is satisfied,
and the results in Theorem~\ref{thm_1hump} hold.

Numerical simulations
are shown in Figures~\ref{fig_predprey_H2_chi} and \ref{fig_predprey_H2eps}.
In the simulations,
the parameters are $(S^0,m,\gamma,c)=(10,1,1,1)$ for \eqref{deq_predprey},
and $(a,b)=(1.5,3)$ and $p(x)$ in \eqref{def_H2}.
Figure~\ref{fig_predprey_H2_chi} shows that
the function $\chi$ has a root $x_0\approx 6.92$,
and it satisfies $\lambda(x_0)<0$.
Hence $\gamma(x_0)$ corresponds
to a stable relaxation oscillation
formed by the globally asymptotically stable periodic orbit of \eqref{deq_predprey}
as $\epsilon\to 0$.
Figure~\ref{fig_predprey_H2eps}(A) shows the periodic orbit $\ell_\epsilon$
for \eqref{deq_predprey} with $\epsilon=0.5$,
and Figure~\ref{fig_predprey_H2eps}(B)
illustrates the trajectory $\gamma(x_0)$ for \eqref{fast_predprey}.
The simulation confirms that the location of $\ell_\epsilon$ is close to $\gamma(x_0)$.
\end{example}

\begin{remark}
The analysis of \eqref{deq_predprey_xy} in this section
can more generally
be applied to systems of the form
\begin{equation}\label{deq_Fpq}\begin{aligned}
  &\dot{x}= q(x)\big(F_\epsilon(x)-y\big),
  \\
  &\dot{y}= y(p(x)-\epsilon),
\end{aligned}\end{equation}
that satisfy
condition \eqref{cond_p},
$q(x)>0$ for $x>0$,
$\frac{p(x)}{q(x)}$ is continuous at $x=0$,
and
\begin{equation}\notag
  F_0(x)\begin{cases}
  >0,& \text{if}\;\; 0\le x<K,
  \\
  <0,& \text{if}\;\; K>x,
  \end{cases}
\end{equation}
for some positive constant $K$.
Similar to Theorem \ref{thm_1hump},
a unique locally orbitally asymptotic stable relaxation oscillation
for system \eqref{deq_Fpq}
in the region filled by a family of heteroclinic orbits
can be obtained
for all small $\epsilon>0$
under assumption \eqref{cond_1hump} with $F(x)=F_0(x)$
and $S^0/\rho$ replaced by $K$.
\end{remark}

\begin{figure}[htb]
{\centering
{\includegraphics[trim= 0cm 0cm 0cm 0cm, clip, width=.48\textwidth]%
{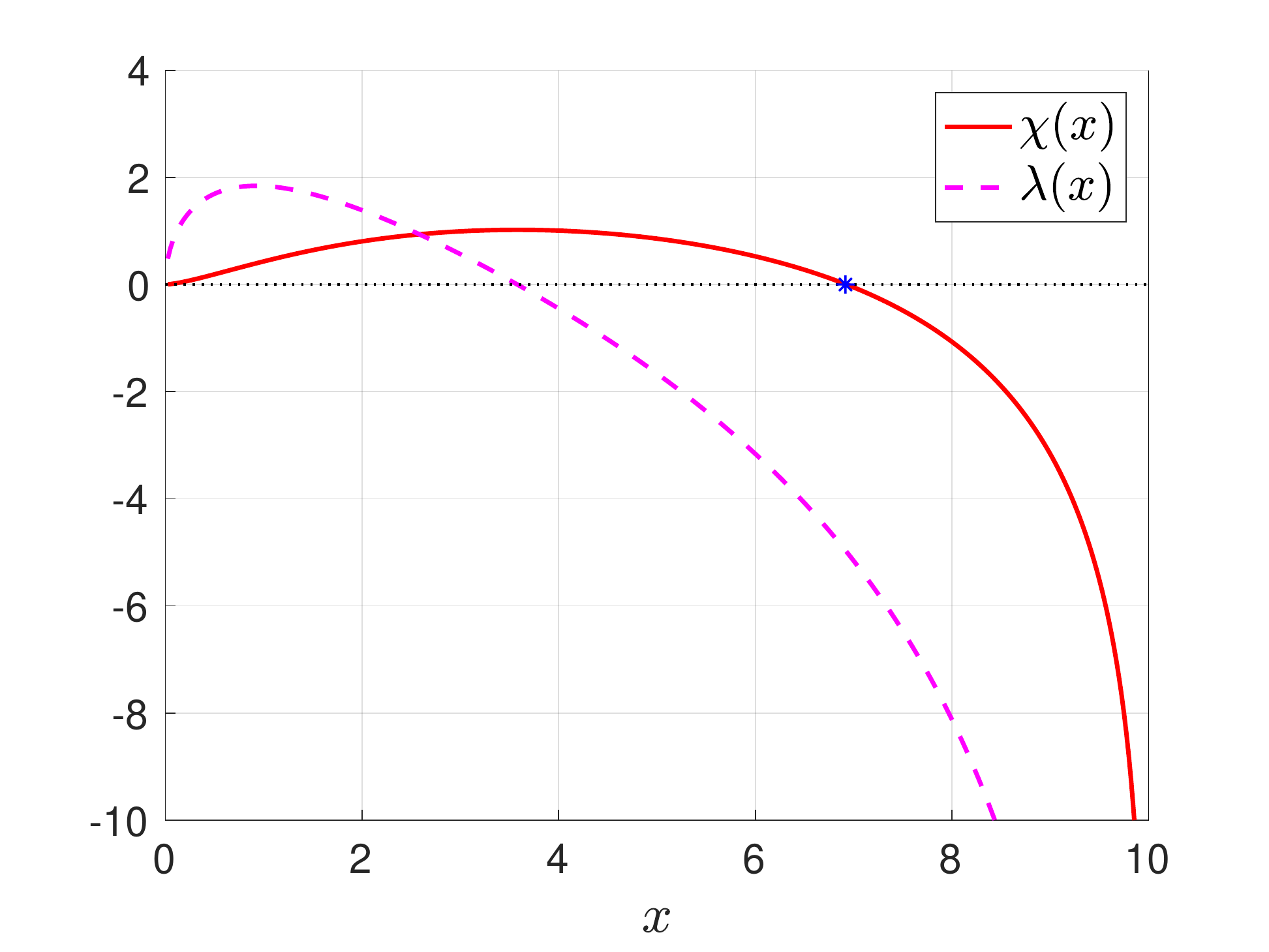}}
\par}
\caption{
Numerical simulations of $\chi$ and $\lambda$
for Example \ref{ex_H2}.
The function $\chi$ has a single root $x_1\approx 6.92$,
with $\lambda(x_1)<0$.
}
\label{fig_predprey_H2_chi}
\end{figure}

\begin{figure}[htb]
{\centering
\begin{tabular}{cc}
{\includegraphics[trim= 1.2cm 0cm 1cm 0cm, clip, width=.48\textwidth]{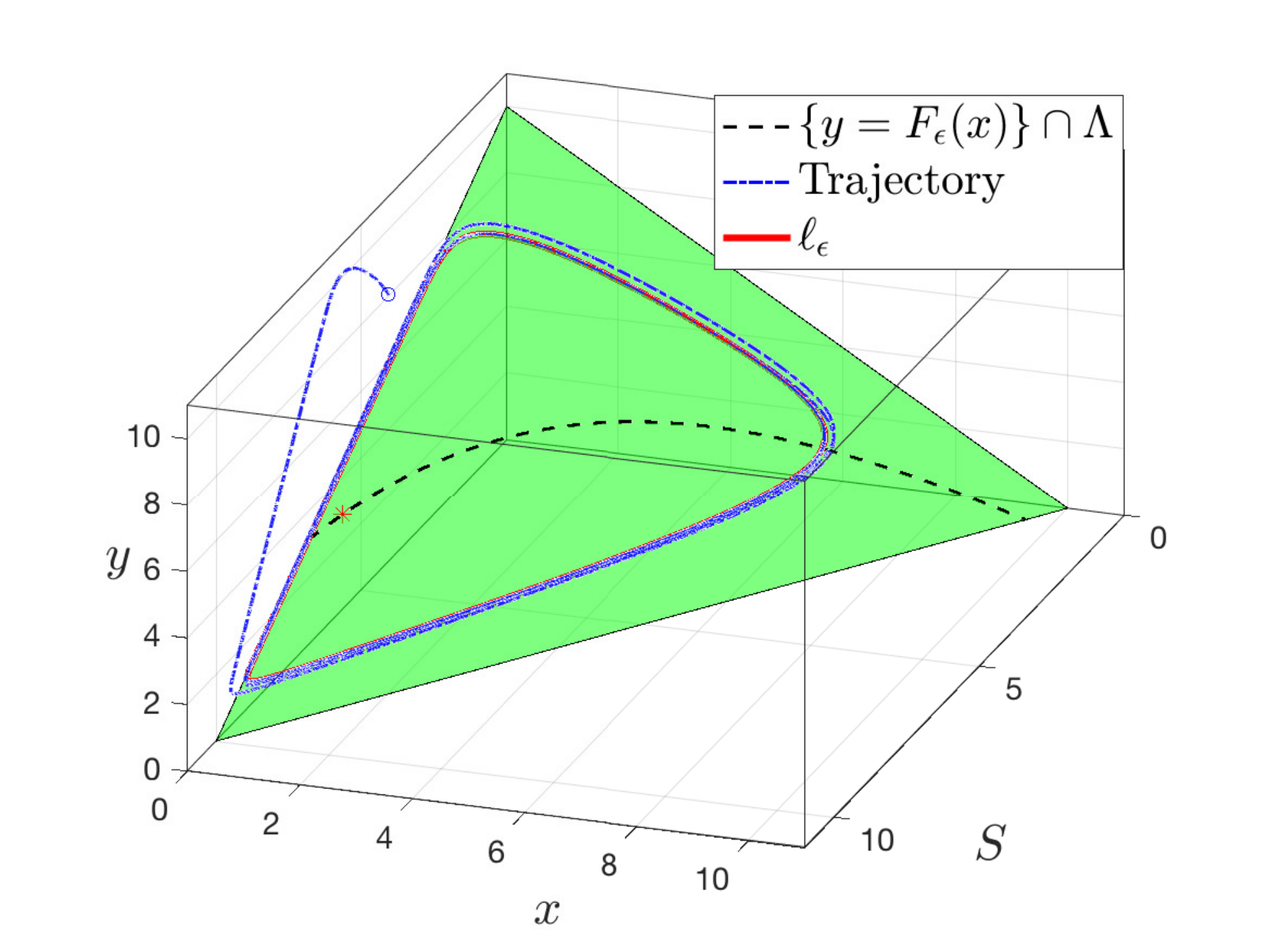}}
&
{\includegraphics[trim= 1.2cm 0cm 1cm 0cm, clip, width=.48\textwidth]{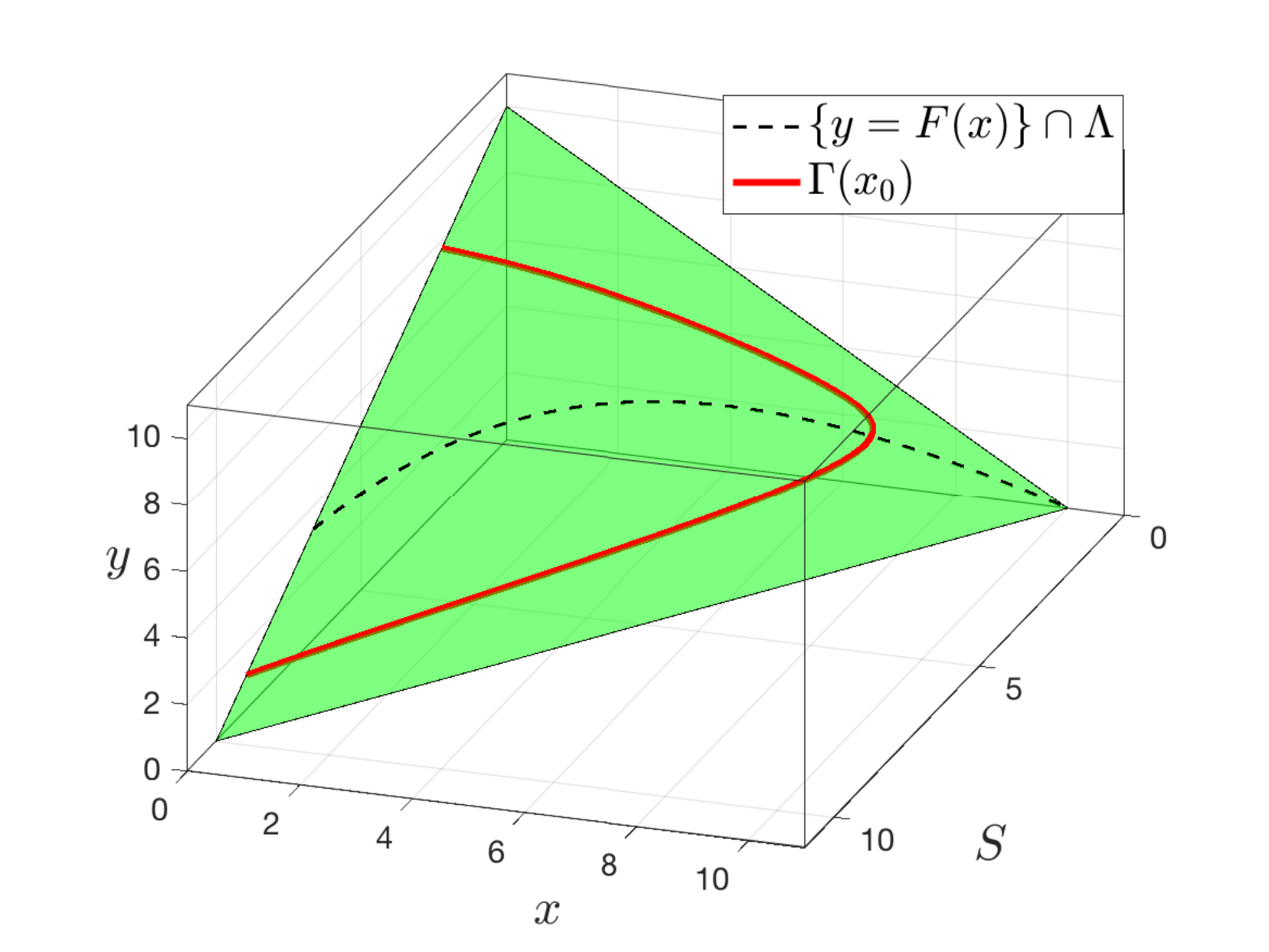}}
\\
(A)
& (B)
\end{tabular}
\par}
\caption{
(A) The trajectory of system \eqref{deq_predprey} for Example~\ref{ex_H2}
with $\epsilon=0.5$
and initial point $(S,x,y)(0)= (6,1,10)$
converges to a periodic orbit $\ell_\epsilon$.
(B) The trajectory $\gamma(x_0)$ of \eqref{fast_predprey},
where $x_0$ is  a root of $\chi$.
The simulation shows that $\ell_\epsilon$ is close to $\gamma(x_0)$
}
\label{fig_predprey_H2eps}
\end{figure}

\section{The Epidemic Model}
\label{sec_SIN}
When $\epsilon=0$,
system \eqref{deq_SIN}  reduces to system \eqref{deq_center_SIN}.
The line $\mathcal{Z}_0=\{(S,I,N): I=0, S=\frac{D}{D+p}N\}$
is a set of equilibria of \eqref{deq_SIN}
in the invariant plane $\{I=0\}$.
Let $N_0$ be the unique value that satisfies $h(\frac{D}{D+p}\bar{N},\bar{N})=a$.
It is known \cite{Graef:1996,Li:2016} that
each point on the segment $\mathcal Z_0\cap \{0<N<N_0\}$
is connected 
to a unique point on the segment $\mathcal Z_0\cap \{N_0<N<N_{\max{}}\}$
 by a heteroclinic orbit of \eqref{deq_center_SIN} (see Figure~\ref{fig_SIN_gamma}).
We define \begin{equation}\notag
  \omega:(N_0,N_{\max})\to (0,N_0)
\end{equation}
such that the point
$(\frac{D}{D+p}\omega(N_1),\omega(N_1),0)$
is the omega-limit point of the heteroclinic orbit,
denoted by $\gamma(N_1)$,
of \eqref{deq_center_SIN}
starting from $(\frac{D}{D+p}N_1,N_1,0)$.

It is also known \cite{Li:2016} that
the invariant manifold $\mathcal{Z}_0$
and the center manifold $W^c_0(\mathcal Z_0)$
of system \eqref{deq_center_SIN}
still exist for the perturbed system \eqref{deq_SIN},
and that the perturbed center manifold
is a global attractor for \eqref{deq_SIN} for each small $\epsilon>0$.
Denote the perturbed manifolds by $\mathcal{Z}_\epsilon$
and $W^c_\epsilon(\mathcal Z_\epsilon)$.
We parametrize the center manifold $W^c_\epsilon(\mathcal Z_\epsilon)$
by $S=\tilde{S}_\epsilon(I,N)$,
and define $\tilde{S}=\tilde{S}_0$.
Then the restriction of system \eqref{deq_SIN} on $W^c(\mathcal Z_\epsilon)$
can be written as 
\begin{equation}\label{deq_SIN_Wc}\begin{aligned}
  &I'=\big(g(\tilde{S}_\epsilon(I,N),N)- a\big)I,
  \\
  &N'= \epsilon f(N) -\alpha I.
\end{aligned}\end{equation}

With $(N,I)$ in \eqref{deq_SIN_Wc} playing the role of $(a,b)$ in \eqref{deq_ab},
and $N_0$ playing the role of $\bar{a}$,
the function $\chi:[N_0,N_{\max})\to \mathbb R$ defined by \eqref{def_chi}
is equal to
\begin{equation}\label{chi_SIN}
  \chi(N_1)
  = \int_{\omega(N_1)}^{N_1} \frac{g\left(\frac{D}{D+p}N,N\right)-a}{f(N)}\;dN.
\end{equation}
By \eqref{lambda_simple2} in Proposition~\ref{prop_lambdaH}
and \eqref{int_gh} in Remark~\ref{rmk_gh},
the function $\lambda:[N_0,N_{\max})\to \mathbb R$
defined by \eqref{def_lambda}
is equal to
\begin{equation}\notag
  \lambda(N_1)
  = \ln\frac{f(N_1)}{f(\omega(N_1))}
  + \int_{\gamma(N_1)}\frac{\partial_I\big(g(\tilde{S}(I,N),N)\big)}{-\alpha}\;dN.
\end{equation} 
That is, \begin{equation}\label{lambda_SIN}
  \lambda(N_1)
  = \ln\frac{f(N_1)}{f(\omega(N_1))}
  -\frac{1}{\alpha} \int_{\gamma(N_1)}\partial_Sg(\tilde{S}(I,N),N)\,\partial_I\tilde{S}(I,N)\;dN.
\end{equation}

\begin{remark}
The function $\chi$ in \eqref{chi_SIN}
is equivalent to the function $\bar{F}$ in \cite{Li:2016}
in the sense that $\chi(N)>0$ if and only of $\bar{F}(N)>N$.
\end{remark}

\begin{theorem}\label{thm_SIN}
Assume that $N_1\in (N_0,N_{\max})$ satisfies $\chi(N_1)=0$
and $\lambda(N_1)\ne 0$,
where $\chi(N)$ and $\lambda(x)$
are defined in \eqref{chi_SIN} and \eqref{lambda_SIN}, respectively.
Then for any sufficiently small $\epsilon>0$,
there is a unique periodic orbit ${\ell}_\epsilon$ of \eqref{deq_SIN}
in a $O(\epsilon)$-neighborhood of $\gamma(N_1)$.
The minimal period of $\ell_\epsilon$, denoted by $T_\epsilon$, satisfies \begin{equation}\label{est_Teps_SIN}
  T_\epsilon= \frac{1}{\epsilon}\left(
    \int_{\omega(N_1)}^{N_1}\frac{1}{f(N)}\;dN+o(1)
  \right)
  \quad\text{as }\epsilon\to 0.
\end{equation}
Moreover, 
${\ell}_\epsilon$ is locally orbitally asymptotically stable if $\lambda(N_1)<0$,
and is orbitally unstable if $\lambda(N_1)>0$.
Conversely, 
if $\chi(N_1)\ne 0$,
then for any point $z_1$ in the interior of the trajectory $\gamma(N_1)$,
there is a neighborhood $U$ of $z_1$
such that no periodic orbit of \eqref{deq_ab}
intersects $U$ for any sufficiently small $\epsilon>0$.
\end{theorem}

\begin{proof}
If $\chi(N_1)\ne 0$, then by Theorem \ref{thm_general},
system \eqref{deq_center_SIN} has no periodic orbit near $\gamma(N_1)$
in $W^c_\epsilon(\mathcal{Z}_\epsilon)$
for any small $\epsilon>0$.
Since $W^c_\epsilon(\mathcal{Z}_\epsilon)$ is a global attractor,
it follows that system \eqref{deq_SIN} has no periodic orbit near $\gamma(N_1)$
in $\mathbb R^3_+$.

Next assume that $\chi(N_1)\ne 0$.
Then by Theorem \ref{thm_general},
system \eqref{deq_center_SIN} has a unique periodic orbit $\ell_\epsilon$
in a $O(\epsilon)$-neighborhood of $\gamma(N_1)$
in $W^c_\epsilon(\mathcal{Z}_\epsilon)$
for every small $\epsilon>0$.
If $\lambda(N_1)>0$,
then $\ell_\epsilon$ is unstable for \eqref{deq_center_SIN},
and therefore is unstable for \eqref{deq_SIN}.
If $\lambda(N_1)<0$,
then $\ell_\epsilon$ is locally orbitally asymptotically stable for \eqref{deq_center_SIN}.
Since $W^c_\epsilon(\mathcal{Z}_\epsilon)$ is a hyperbolic attractor,
$\ell_\epsilon$ is locally orbitally asymptotically stable for \eqref{deq_SIN}.
\end{proof}

\begin{remark}
The period $T_\epsilon$ of the limit cycle $\ell_\epsilon$
is referred to as
{\em interepidemic period} (IEP)
in \cite{Li:2016},
where it was shown numerically that
$T_\epsilon$ is proportional to $1/\epsilon$.
Their observation is consistence with
the asymptotic formula \eqref{est_Teps_SIN}.
\end{remark}

We are not able to determine
the signs of $\chi$ and $\lambda$ in \eqref{chi_SIN} and \eqref{lambda_SIN} analytically.
Nonetheless,
we are able to compute $\chi$ and $\lambda$ numerically.
A numerical difficulty in the computation
is to approximate $\partial_I\tilde{S}$,
since there is no explicit formula for $\tilde{S}$.
We implement the following algorithm to compute $\partial_I\tilde{S}$:
Fix a small number $\delta>0$
and large integers $T$ and $M$.
For $N_1\in (N_0,N_{\max})$,
denote by $(S,I,N)(t;N_1)$ the solution of
\eqref{deq_center_SIN}
with initial value
$(S,I,N)(0)=(\frac{D}{D+p}N_1,0,N_1)+\delta \vec{v}$,
where $\vec{v}$ is an eigenvector corresponding
to the unstable eigenvalue of 
the linearization of \eqref{deq_center_SIN}
at $(\frac{D}{D+p}N_1,0,N_1)$,
so that the forward trajectory of this solution is near
the heteroclinic orbit
with the alpha-limit point $(\frac{D}{D+p}N_1,0,N_1)$.
Let $N_k$, $k=1,\dots,M$, be a grid of the interval $[N_0+\delta,N_{\max}]$,
and let $t_j$, $j=0,1,\dots, TM$, be a grid of the interval $[0,T]$.
Define \begin{equation}\notag
  u_k^j=u(t_j;N_k)
  \quad\text{for}\;\; 1\le k\le M,\; 1\le j\le TM,\; u=S,I,N,
\end{equation}
Define $\Delta_x u_k^j= u_{k+1}^j-u_k^j$
and $\Delta_tu_k^j=u_{k}^{j+1}-u_k^j$.
Then the numerical approximations
of $\partial_I\tilde{S}$ and $\partial_N\tilde{S}$
are \begin{equation}\label{partialS}\begin{aligned}
  \begin{bmatrix}
      (\partial_I\tilde{S})_{k}^j
      &
      (\partial_N\tilde{S})_{k}^j
  \end{bmatrix}
  =
  \begin{bmatrix}
      \Delta_t S_{k}^{j}
      &
      \Delta_x S_{k}^{j}
  \end{bmatrix}
  \begin{bmatrix}
      \Delta_tI_k^j& \Delta_xI_k^j
      \\[.5em]
      \Delta_tN_k^j& \Delta_xN_k^j
  \end{bmatrix}^{-1}.
\end{aligned}\end{equation}
We use this approximation for $\partial_I\tilde{S}$
to evaluate formula \eqref{lambda_SIN} for $\lambda$.

\begin{example}
\label{ex_SIN1}
Following Li et al.\ \cite[Section 5.1, Case 1]{Li:2016},
we consider $g(S,N)=\frac{\beta S}{m+S}$
and parameters
$D=0.2$, $p=0.01$, $\alpha=0.048$, $\beta=1$, $\gamma=0.75$,
$m=0.1$ and $N_{\max}=400$.
It was proved in \cite{Li:2016} that,
form small $\epsilon>0$,
system \eqref{deq_SIN} has
a stable periodic orbit
and the positive equilibrium is unstable.
Our numerical simulation,
illustrated in Figure~\ref{fig_SIN1}(A),
shows that $\chi$ defined by \eqref{chi_SIN}
has a root $N_1\approx 377.01$
with $\lambda(N_1)\approx -4.11<0$.
Hence $\gamma(N_1)$ corresponds to a stable relaxation oscillation.
A trajectory with initial data $(S,I,N)=(60,2,120)$
and $\epsilon=10^{-5}$
is shown in Figure~\ref{fig_SIN1}(B).
The trajectory first
is attracted by the slow manifold $W_\epsilon^c(\mathcal{Z}_\epsilon)$,
and then follows the dynamics on $W_\epsilon^c(\mathcal{Z}_\epsilon)$
to approach the periodic orbit near $\gamma(N_1)$.
\end{example}

In the next example
we demonstrate that
by varying the perturbation term $\epsilon f(N)$,
system \eqref{deq_SIN} can obtain two relaxation oscillations.
The idea is that
the positive function $f(N)$
effects the magnitude of the integrand
in formula \eqref{chi_SIN} of $\chi(N)$,
and hence the function $\chi(N)$ can gain extra roots by deforming $f(N)$.
This method conceptually can be used
to construct an arbitrary number of relaxation oscillations.

\begin{example}
\label{ex_SIN2}
We replace $f(N)$ defined by \eqref{SIN_f}
with \begin{equation}\notag
  f_1(N)=f(N)-c_1\exp(-c_2(N-c_3))
\end{equation}
with $(c_1,c_2,c_3)=(60,0.04,90)$.
The function $f_1(N)$
differs from $f(N)$ 
essentially only
in a small interval to the right of $N_0$ (see Figure~\ref{fig_SIN_f}).
From our numerical simulation,
as shown in Figure~\ref{fig_SIN_ex2}(A),
the function $\chi$ defined by \eqref{chi_SIN}
with $f$ replaced by $f_1$
has two roots,
$N_1\approx 156.89$ and $N_2\approx 342.18$,
with $\lambda(N_1)\approx 1.06>0$
and $\lambda(N_2)\approx -2.48<0$.
Hence $\gamma(N_1)$ corresponds to an unstable relaxation oscillation
for system \eqref{deq_SIN}
with $f$ replaced by $f_1$,
and $\gamma(N_2)$ corresponds to a stable relaxation oscillation.
Some trajectories
for the system with $\epsilon=10^{-5}$
are shown in Figure~\ref{fig_SIN_ex2}(B).
\end{example}

\begin{remark}
Example \ref{ex_SIN2}
is consistent with Li et al.\ \cite[Section 5.1, Case 2]{Li:2016},
in which two periodic orbits were observed.
It was proved in \cite{Li:2016} that,
for some parameters,
system \eqref{deq_SIN} exhibits
a stable periodic orbit
while the positive equilibrium is asymptotically stable,
which implies that there must be at least one unstable periodic orbit.
It was commented in \cite[Section 4.3]{Li:2016} that, in general,
the unstable periodic orbit is not necessarily a relaxation oscillation
but a small periodic orbit through a subcritical Hopf bifurcation.
Using Theorem~\ref{thm_SIN},
we are able to confirm that
there is an isolated unstable periodic orbit for small $\epsilon>0$
that forms a relaxation oscillation.
\end{remark}

\begin{figure}[htb]
{\centering
\begin{tabular}{cc}
{\includegraphics[trim = 2.4cm 6.2cm 1.7cm 6.6cm, clip, width=.48\textwidth]%
{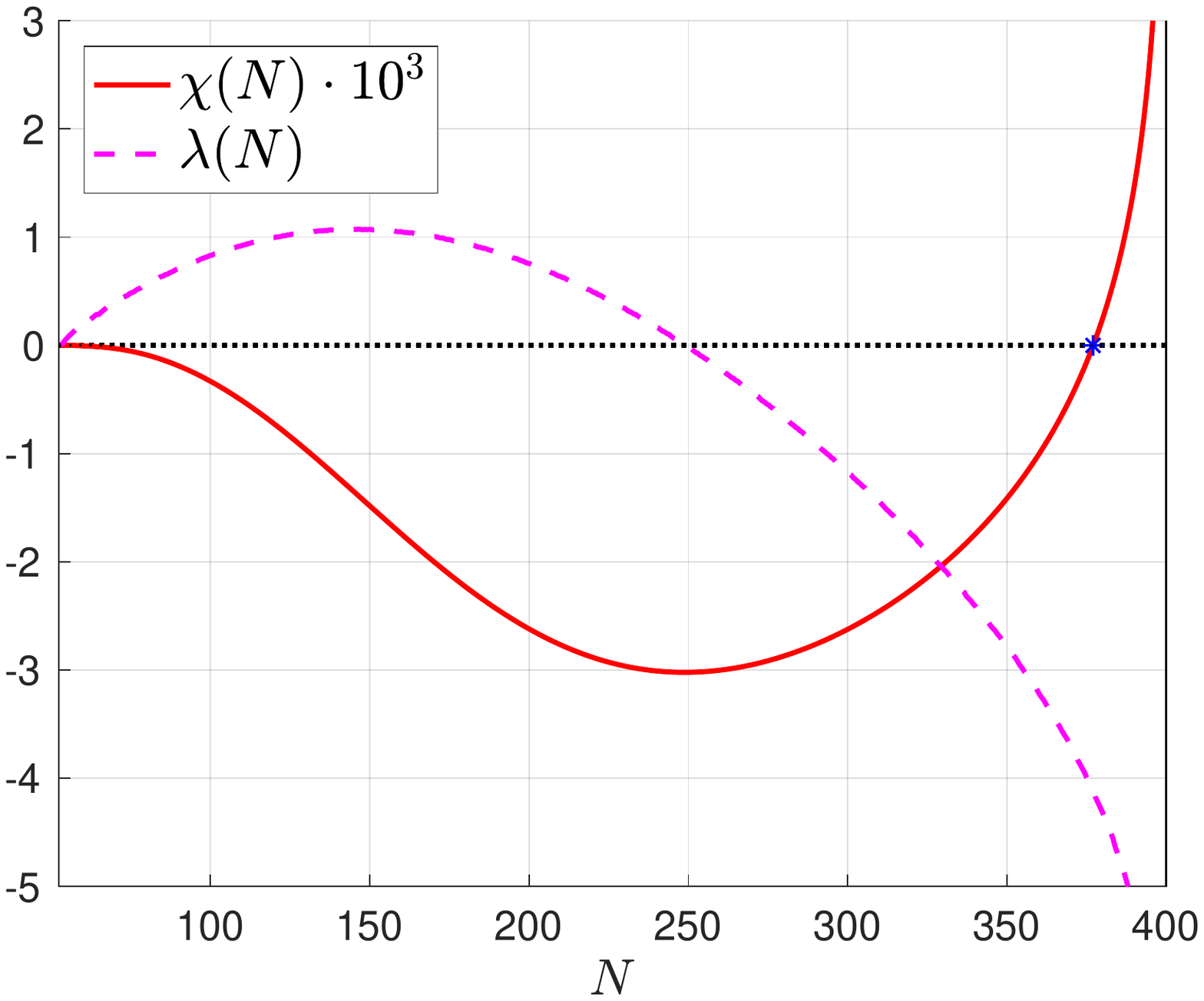}}
&
{\includegraphics[trim= 1.2cm 0cm 1cm 0cm, clip, width=.48\textwidth]{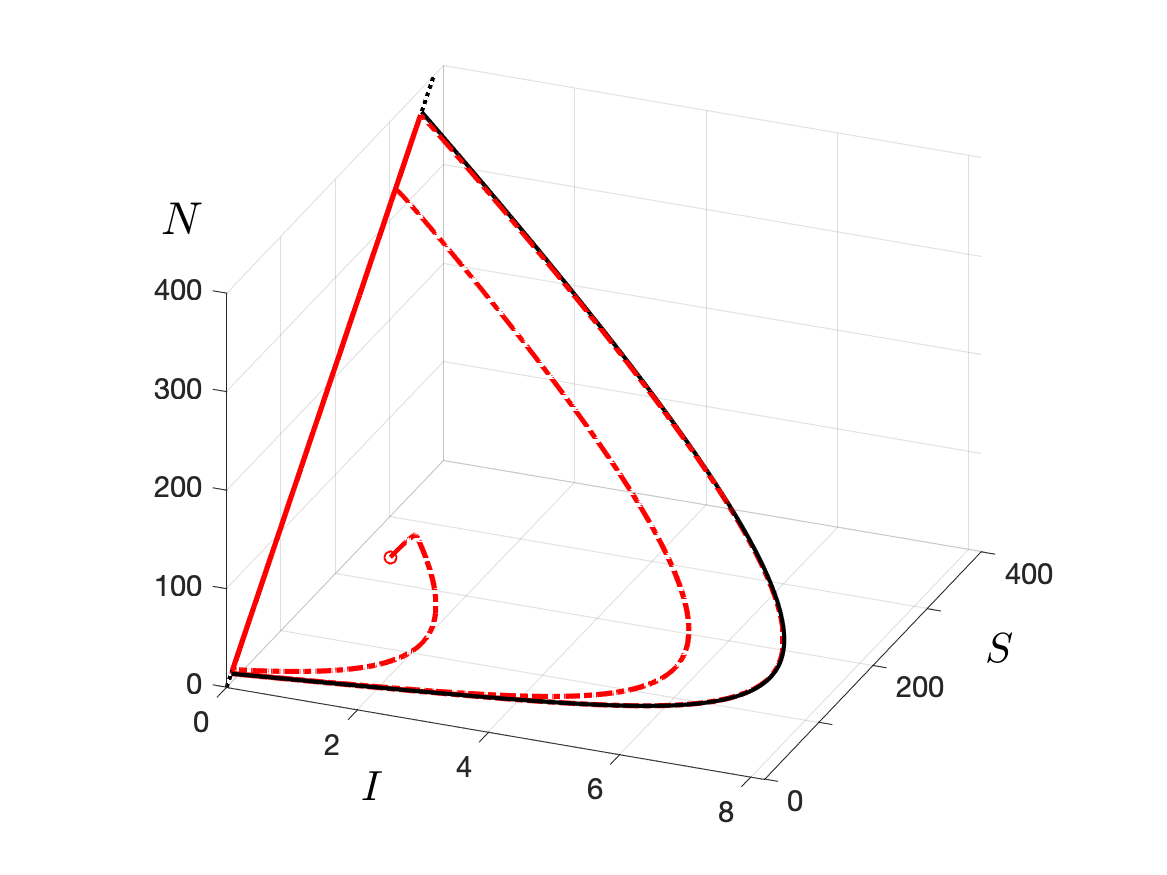}}
\\
(A)
& (B)
\end{tabular}
\par}
\caption{
(A) For Example~\ref{ex_SIN1},
the function $\chi$ has a root $N_1\approx 377.01$
with $\lambda(N_1)\approx -4.11<0$.
(B) 
The dashed curve is a trajectory for
the system with $\epsilon=10^{-5}$
with the initial condition
$(S,I,N)(0)=(60,2,120)$,
and the solid curve is the singular orbit $\gamma(N_1)$.
The simulation shows that
the trajectory
approaches a periodic orbit near $\gamma(N_1)$.
}
\label{fig_SIN1}
\end{figure}

\begin{figure}[htb]
{\centering
{\includegraphics[trim = 2.4cm 6.2cm 1.7cm 6.6cm, clip, width=.48\textwidth]%
{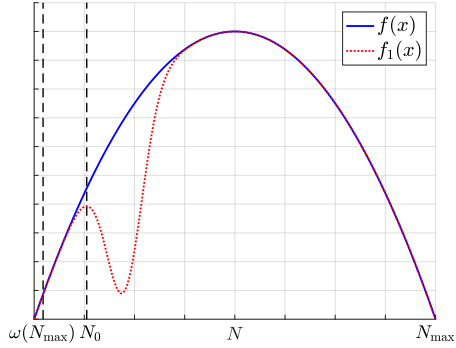}}
\par}
\caption{
The perturbation term $\epsilon f(N)$
with $f(N)=N(1-N_{\max})$ 
in Example \ref{ex_SIN1}
is replaced by $\epsilon f_1(N)$
with $f_1(N)=f(N)-c_1\exp(-c_2(N-c_3))$
in Example \ref{ex_SIN2}.
Essentially $f_1$ is obtained
by dropping the value of $f$ in a small interval right to $N_0$.
}
\label{fig_SIN_f}
\end{figure}

\begin{figure}[htb]
{\centering
\begin{tabular}{cc}
{\includegraphics[trim = 2.4cm 6.2cm 1.7cm 6.6cm, clip, width=.48\textwidth]{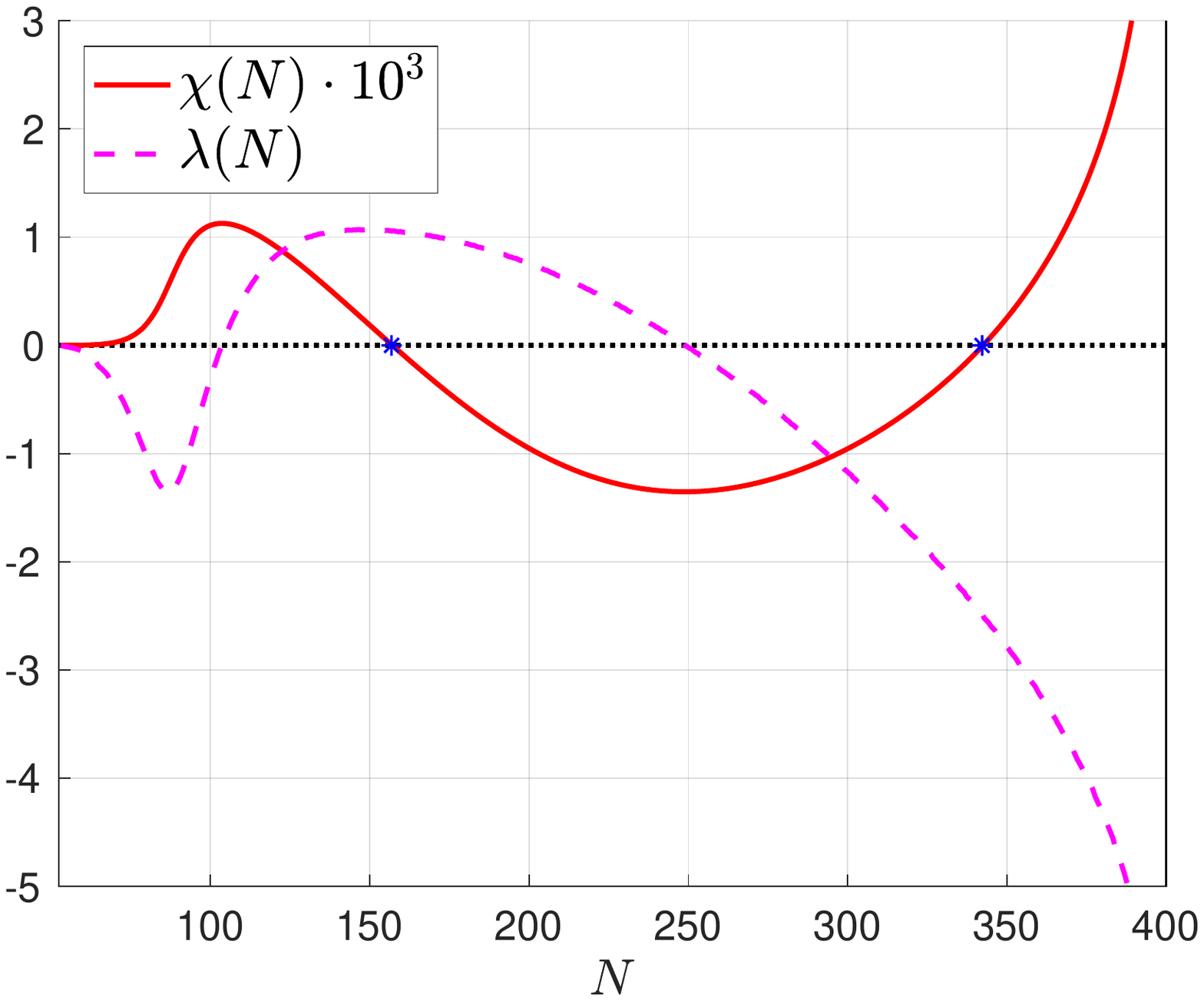}}
&
{\includegraphics[trim= 1.2cm 0cm 1cm 0cm, clip, width=.48\textwidth]{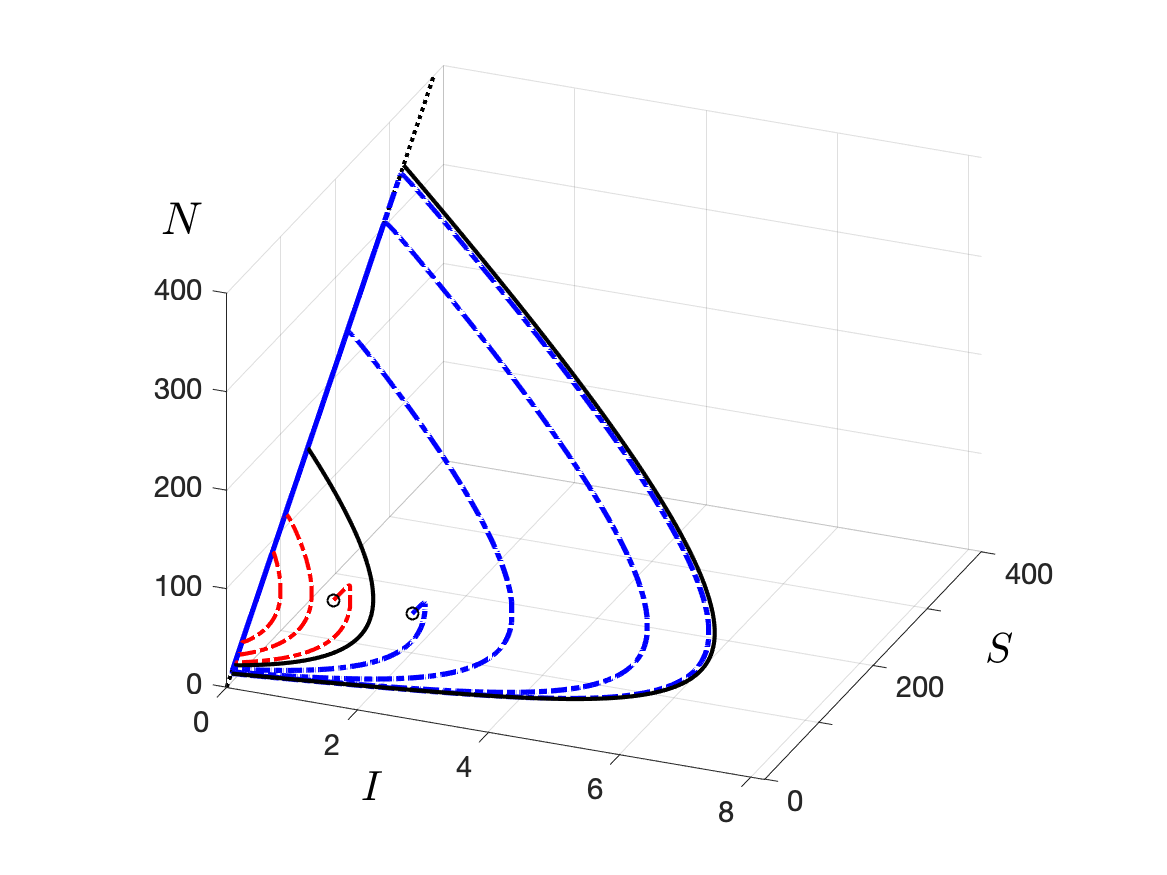}}
\\
(A)& (B)
\end{tabular}
\par}
\caption{
(A) For Example~\ref{ex_SIN2},
the function $\chi$ has two roots,
$N_1\approx 156.89$ and $N_2\approx 342.18$,
with $\lambda(N_1)\approx 1.06>0$
and $\lambda(N_2)\approx -2.48<0$.
(B) 
The dashed curves are trajectories for
the system with $\epsilon=10^{-5}$,
and the solid curves are the singular orbits
$\gamma(N_1)$ and $\gamma(N_2)$.
The simulation shows that
the trajectory
with the initial condition
$(S,I,N)(0)=(40,2.5,80)$
approaches a periodic orbit near $\gamma(N_2)$
while trajectory
with the initial condition
$(S,I,N)(0)=(40,1.3,80)$
approaches the interior equilibrium.
Near $\gamma(N_1)$ is an unstable periodic orbit.
}
\label{fig_SIN_ex2}
\end{figure}

\section{Discussion}
\label{sec_discussion}

In this paper 
we derived a criterion
to determine the location and stability of relaxation oscillations
for the planar system \eqref{deq_ab}
under assumption~$\mathrm{(\hyperref[cond_H]{H})}$.
In Theorem \ref{thm_general},
characteristic functions $\chi(s)$ and $\lambda(s)$
were obtained from
the trajectory $\gamma(s)$
of the limiting system with $\epsilon=0$
and satisfy that
\begin{center}
  \begin{tabular}{l}
  (i)\; $\chi(s_0)\ne 0$
  \;$\Rightarrow$\;
	no periodic orbit passes  near $\gamma(s_0)$ for all $0<\epsilon\ll 1$.
  \\[.4em]
  (ii) $\chi(s_0)=0$ and $\lambda(s_0)\ne 0$
  \;$\Rightarrow$\;
  \text{$\gamma(s_0)$ admits a relaxation oscillation as $\epsilon\to 0$.}
\end{tabular}
\end{center}
In the latter case, the sign of $\lambda(s_0)$
determines the stability of the limit cycles.
The proof of the theorem involves
geometric singular perturbation theory
and a variation of Floquet Theory.
In Propositions \ref{prop_lambdaFH},
\ref{prop_lambdaH} and \ref{prop_lambdaG},
simplified expressions of $\chi$ and $\lambda$
were provided
for the case
that some terms in the system
satisfy certain forms.

We applied this criterion to
two systems using two different techniques.
By relating $\chi(s)$
with a line integral on the heteroclinic orbit,
in Theorem \ref{thm_1hump}
we derived a condition
under which
the chemostat predator-prey system \eqref{deq_predprey}
has a unique limit cycle.
By varying the perturbation term,
we demonstrate
in Examples \ref{ex_SIN1} and \ref{ex_SIN2}
that 
the the epidemic model \eqref{deq_SIR}
can have a prescribed number
of relaxation oscillations.
In general,
it is a numerical challenge
to compute unstable periodic orbits in three-dimensional systems.
Our characteristic functions
provide a way to locate both stable and unstable periodic orbits.
 
\section*{Acknowledgements}
The research for this paper was completed while
TH was a Fuqua Research Assistant Professor at the University of Miami.
GW was supported by the Natural Sciences and Engineering Council (NSERC)
of Canada Discovery Grant Accelerator Supplement.
The authors would like to thank
the anonymous referees
for their constructive comments.
TH thanks Prof.~Zhisheng Shuai
for sharing his insights on
the interepidemic period for the epidemic model.


\end{document}